\documentclass[11pt]{article}

\usepackage[letterpaper,margin=0.84in]{geometry}

\usepackage{amsmath,amsthm}

\usepackage{tgtermes}
\usepackage{newtxtext}
\usepackage{newtxmath}
\usepackage{bm}

\usepackage{graphicx}
\graphicspath{{./}}

\usepackage{booktabs}
\usepackage{adjustbox}
\usepackage{longtable}
\usepackage{multirow}
\usepackage{enumitem}

\makeatletter
\let\arxiv@origfbox\fbox
\renewcommand{\fbox}[1]{%
  \begingroup
  \sbox\@tempboxa{#1}%
  \ifdim\wd\@tempboxa>\dimexpr\linewidth-2\fboxsep-2\fboxrule\relax
    \arxiv@origfbox{\resizebox{\dimexpr\linewidth-2\fboxsep-2\fboxrule\relax}{!}{#1}}%
  \else
    \arxiv@origfbox{#1}%
  \fi
  \endgroup
}
\makeatother

\usepackage{algorithm}
\usepackage{algpseudocode}

\usepackage{tikz}

\usepackage{natbib}
\bibpunct{(}{)}{,}{a}{}{,}

\usepackage[colorlinks=true,
            citecolor=blue,
            urlcolor=blue,
            linkcolor=blue]{hyperref}

\theoremstyle{plain}

\newtheorem{proposition}{Proposition}

\theoremstyle{definition}

\theoremstyle{remark}
\newtheorem{remark}{Remark}

\renewenvironment{proof}[1]{%
  \par\noindent\textbf{#1.}\quad\ignorespaces
}{\hfill\qed\par\medskip}

\title{Learning Optimization Proxies for Sequential Contextual Stochastic
       Programs: An Order Fulfillment Application}

\author{%
  Tinghan Ye$^{1}$\thanks{Corresponding author: \texttt{joe.ye@gatech.edu}}, \,
  Shuaicheng Tong$^{1}$, \,
  Changkun Guan$^{1}$, \,
  Beste Basciftci$^{2}$, \,
  Pascal Van Hentenryck$^{1}$
  \\[1ex]
  \small $^{1}$H.\ Milton Stewart School of Industrial and Systems Engineering,
  Georgia Institute of Technology, Atlanta, GA 30332\\
  \small $^{2}$Tippie College of Business, University of Iowa,
  Iowa City, IA 52242%
}

\date{}

\begin{document}
\maketitle

\begin{abstract}
\noindent Sequential contextual stochastic programs model real-time decision systems in which each time epoch commits to an action under uncertainty whose consequences propagate into future decisions. In many practical contexts, these programs require obtaining solutions rapidly as new information becomes available. These problems can be represented through scenario approximations to be solved by off-the-shelf optimization solvers, which achieve high decision quality offline but typically run in seconds to minutes per instance, falling short of the sub-second responses that peak periods of planning require. This paper develops a learning-based optimization proxy: a scenario-embedded neural network trained offline on solver-generated labels, paired online with a decoder that enforces feasibility, replacing the per-epoch solve with a single forward pass. The framework is specialized to omnichannel order fulfillment, where each arriving order requires a sub-second assignment of products to distribution centers and carrier services under stochastic delivery times and future demand. A two-stage contextual stochastic program is introduced to formulate this problem, and its contextual sample average approximation (C-SAA) supplies the offline labels, while a composite training loss combines label imitation, a constraint-violation penalty, and self-supervised cost alignment. In a calibrated simulator built from JD.com transactional records, a detailed computational study is provided. The proxy reduces decision latency by roughly $2{,}800\times$ relative to the online finite-sample C-SAA reference and improves over it by $3.3\%$ in realized fulfillment cost. Relative to established fulfillment policies, the proxy lowers total realized cost by at least $10.7\%$ and roughly halves the late-delivery rate. It improves both fulfillment cost and on-time delivery, the service dimension that drives customer satisfaction, at a runtime that stays suitable for peak workloads.

\medskip
\noindent\textbf{Keywords:} E-commerce Fulfillment; Contextual Stochastic
Optimization; Machine Learning for Combinatorial Optimization; Real-Time
Decision Making; Optimization Proxy.

\medskip
\noindent\textbf{Code:} \url{https://github.com/Joeyetinghan/contextual-fulfillment-proxy}.
\end{abstract}

\section{Introduction} \label{sec:intro}

Many real-time decision systems share a common abstraction: at each epoch the decision maker observes contextual covariates, commits to an action under uncertainty, and carries the consequences into future decisions. The same structure appears across many real-time operations. In real-time electricity dispatch, system operators clear security-constrained economic dispatch every five minutes to rebalance generation against fluctuating load and prices \citep{chen2024economic}. In ride-hail matching, platforms assign drivers to riders as requests arrive under demand and travel-time uncertainty \citep{ozkanWard2020dynamicmatching}. In dynamic pricing, sellers revise prices as demand and changing circumstances are observed over time, where digital technologies enabled continuously adjusting prices in many business contexts \citep{denBoer2015dynamicpricing}. These systems run at large scale under tight latency budgets. For instance, Uber recorded roughly $11.3$ billion trips in 2024 \citep{uber2024results}, and Alibaba reported a peak of $583{,}000$ orders created per second during its 2020 Singles' Day event \citep{alibaba2020singlesday}. When a high-quality decision cannot be produced in time, operators fall back on simpler heuristics that forgo the value of optimization, and at these volumes even small per-order losses accumulate rapidly.

Sequential contextual stochastic programs formalize this setting by representing the underlying uncertainty through a conditional distribution based on the covariate information. For the solution of these problems, a standard approach approximates this distribution by a finite scenario set and solves the resulting stochastic program under these scenarios at every epoch by utilizing an off-the-shelf optimization solver. While this delivers high decision quality offline, the scenario-based reformulation typically runs in seconds to minutes per instance depending on its size, falling short of the sub-second responses that peak periods demand, which require the development of novel computationally efficient solution techniques.

Omnichannel order fulfillment in large e-commerce networks is a representative problem setting where the fast response times to satisfy the order assignments are critical for achieving effective operations of these systems \citep{acimovic2019fulfillment}. Specifically, each arriving order must be assigned to a distribution center and carrier service that together balance shipping cost, delivery-time risk, multi-unit consolidation discounts, and the opportunity cost of consuming inventory needed by future orders. Leveraging context information which can include factors such as order characteristics, state of the distribution and carrier network, weather information and current demand, can be critical in this regard to represent the underlying uncertainties such as delivery times accurately. To this end, although contextual stochastic optimization (CSO) has lowered the realized cost of these decisions by integrating distributional forecasts of delivery times \citep{ye2025contextual}, extending CSO to sequential online environments remains computationally challenging \citep{sadana2025survey}.

Concretely, in the JD.com transactional setting used in this paper for the order fulfillment problem \citep{shen2024jd}, peak hours of a typical day see hundreds of order arrivals per minute, with even higher rates during holiday seasons, while the contextual sample average approximation (C-SAA) reformulation \citep{verweij2003sample} typically takes seconds to minutes per order to solve to a tight optimality gap \citep{ye2025contextual}. Solver-based C-SAA is therefore unavailable when peak periods demand sub-second responses, and operational systems fall back on heuristics that ignore the modeled uncertainty.

This paper develops a learning-based optimization proxy for sequential contextual stochastic programs and specializes it to omnichannel order fulfillment. The proxy is a scenario-embedded hierarchical neural network trained offline on solver-based labels, paired at deployment with a dedicated decoder that returns a feasible plan; together they replace the per-epoch solve with a single forward pass. Unlike prior optimization proxies that do not explicitly condition on the contextual uncertainty distribution \citep{yuan2022reinforcement, qi2023practical, yilmaz2024non, ojha2025outbound, vanhentenryck2025optimization}, the proposed architecture is explicitly scenario-aware, capturing both immediate stochastic costs and the sequential impact of resource consumption.

The contributions of this paper are four-fold:
\begin{enumerate}[leftmargin=*,topsep=0pt,itemsep=0pt,parsep=0pt,partopsep=0pt]
\item {\em A scenario-conditioned optimization proxy for sequential contextual stochastic programs}: This model maps contextual features and scenarios to decisions in a single forward pass, ensuring feasibility via a lightweight decoder and utilizing a multi-term loss function that integrates solver imitation, constraint penalties, and self-supervised cost-alignment.
\item {\em A two-stage contextual stochastic model for sequential omnichannel fulfillment}: This formulation explicitly ties each arriving order's assignment to the inventory available for future orders. To address this problem, a tailored proxy is introduced, featuring a hierarchical neural network over distribution centers and carrier services that generates scenario-conditioned scores. An inventory-weighted decoder then converts these scores into a feasible plan in a single pass.
\item {\em Theoretical guarantees for the fulfillment proxy}: The proposed inventory-weighted decoder is shown to yield a feasible decision in time log-linear in the size of the fulfillment network (the number of distribution centers and carrier services) per product line. Furthermore, its expected cost upper-bounds the optimal two-stage objective, and its out-of-sample value admits unbiased, consistent sample-mean and sample-variance estimators.
\item {\em Empirical evaluation via an industry-calibrated simulator}: Using a simulation framework built from JD.com transactional records and proprietary-calibrated carrier services, the operational impact of the proxy is demonstrated. Compared to an online solver-based baseline utilizing the same fitted forecaster, the proposed approach reduces per-order decision latency by approximately $2{,}800\times$ while lowering realized fulfillment costs by $3.3\%$. Beyond computational efficiency and better out-of-sample performance, the proxy improves operational outcomes. Relative to established fulfillment policies, it lowers total realized fulfillment cost by at least $10.7\%$ and roughly halves the late-delivery rate of the strongest classical baselines. Fewer late deliveries raise on-time service, the dimension that drives customer satisfaction.
\end{enumerate}

\noindent
The remainder of the paper is organized as follows. Section~\ref{sec:lit_review} reviews related work. Section~\ref{sec:problem_formulation} formalizes the two-stage contextual stochastic model and its C-SAA reformulation. Section~\ref{sec:proxy-framework} develops the scenario-embedded proxy and instantiates it for omnichannel fulfillment. Section~\ref{sec:computational_study} reports the computational study on JD.com data. Section~\ref{sec:conclusion} concludes the paper and discusses future directions.

\section{Literature Review} \label{sec:lit_review}
This section surveys three relevant domains which this paper provides contributions: contextual stochastic optimization, machine learning for stochastic and combinatorial optimization, and dynamic e-commerce fulfillment.

\subsection{Contextual Stochastic Optimization}
Classical stochastic programming typically represents uncertainty with an empirical distribution estimated from historical data, ignoring any auxiliary information that might be available.
In contrast, contextual stochastic optimization (CSO) uses concurrent auxiliary information, or covariates, to construct conditional distributions and thereby produce context-dependent
decisions \citep{sadana2025survey}. 
CSO methodologies fall broadly into three streams: sequential
predict-and-optimize, integrated learning and optimization, and direct
decision-rule learning. Sequential predict-and-optimize methods train a
predictive model and pass its output to a downstream solver \citep{bertsimas2020predictive}. Integrated approaches, or
decision-focused learning, instead train predictive models directly with
respect to downstream decision quality \citep{donti2017task,wilder2019melding}.
However, in settings with non-linear and asymmetric costs, such as e-commerce
fulfillment with severe late-delivery penalties, relying only on point
forecasts may still be insufficient.

To address this limitation, recent CSO methods increasingly use conditional
scenario generation to represent predictive uncertainty
\citep{kannan2025data,jia2025scenario,qi2025integrated}. In omnichannel fulfillment, \citet{ye2025contextual} established the benefit of distribution-aware C-SAA relative to point-forecast and empirical-SAA baselines when delivery uncertainty is modeled explicitly. The present paper extends this from a static offline setting to an online two-stage contextual
stochastic program that also captures the future opportunity cost of inventory
consumption. To satisfy tight latency requirements, the proposed framework adopts direct decision-rule learning, training a parameterized policy that maps contextual states directly to fulfillment decisions. While related to sequential control, direct reinforcement learning is less attractive here due to hard operational constraints and the availability of offline optimization labels.

\subsection{Machine Learning for Stochastic and Combinatorial Optimization}
The intersection of machine learning and operations research has yielded diverse methodologies for combinatorial and stochastic optimization, ranging from accelerating exact heuristics \citep{gasse2019exact} to differentiable categorical relaxations \citep{jang2017categorical}. In multi-stage stochastic programming, learning-based proxies address intractability primarily through direct solution prediction, value function approximation, and optimization-integrated learning \citep{rosemberg2025efficiently}. For instance, Neur2SP \citep{patel2022neur2sp} approximates the expected second-stage value function by pre-training a neural network that is subsequently encoded as fixed constraints in a downstream mixed-integer program. \citet{chan2025machine} avoids pre-training by embedding the second-stage cost approximation directly into the optimization formulation alongside the first-stage decisions, yielding new theoretical optimality bounds.

A directly relevant thread is constrained optimization learning \citep{vanhentenryck2025optimization}, which imitates complex optimization oracles for sub-second inference in autonomous fleet control and vehicle relocation \citep{yuan2022reinforcement, jungel2025learning}, recurrent dual-sourcing inventory control \citep{bottcher2023control}, and end-to-end load planning \citep{ojha2025outbound}. However, these works primarily study deterministic settings or rely on point-forecast approximations, leaving scenario-aware policies for sequential contextual stochastic programs largely unexplored. The proposed proxy differs from these threads on three axes. First, a permutation-invariant DeepSets-style encoder consumes a sampled scenario set as a first-class input rather than a point forecast. Furthermore, solver-label imitation is augmented by a constraint-violation penalty and a self-supervised cost-alignment term, rather than the decision-focused gradients of the respective learning setting \citep{donti2017task, wilder2019melding}. Additionally, decisions are first-stage and per-epoch in a sequential inventory-coupled setting, rather than the static second-stage value-function approximations of \citet{patel2022neur2sp} and \citet{chan2025machine}.

\subsection{Dynamic E-Commerce Fulfillment}
Omnichannel fulfillment requires assigning orders to facilities and carrier services to balance costs, consolidation, and service levels \citep{acimovic2019fulfillment, vasconcelos2026order}. While rolling-horizon policies and dynamic inventory heuristics improve myopic decisions by addressing future demand uncertainty and capacity limits
\citep{acimovic2015making,jasin2015lp,jia2026enabling}, they typically assume deterministic delivery times. Conversely, recent data-driven studies emphasize timeliness by explicitly modeling delivery uncertainty
\citep{salari2022real,ye2025contextual,ye2025conformal, faulkner2026uncertainty}. Because existing models have yet to integrate both uncertainty streams simultaneously, formulating a sequential model that jointly captures future demand and delivery uncertainties constitutes a primary contribution of this paper. Beyond these modeling contributions, the paper develops a learning-based solution approach for this problem. A scenario-aware optimization proxy returns fulfillment decisions in a single forward pass and runs several orders of magnitude faster than solver-based stochastic optimization. Existing order-fulfillment methods do not develop such a proxy, so the need for fast, scenario-aware solutions in this setting remains as an open challenge.

\section{Problem Formulation} \label{sec:problem_formulation}

Many e-commerce firms face an online order-fulfillment problem in their daily operations. Customers place orders continuously throughout the day, and each order must be served quickly under tight delivery promises. An order consists of one or more products, each requested in some quantity, together with the customer's destination and order attributes such as product categories, brands, order type, and applicable discounts. When an order arrives, the firm observes its contents, the on-hand inventory at each distribution center, the available carrier services and their shipping costs, and contextual signals predictive of delivery times and future demand. The firm must then immediately decide, for each product, which distribution center fulfills it and which carrier service ships it. Because these decisions consume inventory that later orders also draw on, each assignment shapes the cost and feasibility of fulfilling future orders. The remainder of this section formalizes this setting.

The online order fulfillment is formulated over a planning horizon $\mathcal{T} = \{1, 2, \ldots, T\}$, where each epoch $t \in \mathcal{T}$ corresponds to the arrival of a customer order. The fulfillment network consists of a set of locations/distribution centers (DCs) $\mathcal{J}$, a set of products $\mathcal{I}$, and a set of carrier-services $\mathcal{K}$. Each order requests quantities $\mathbf{q}^t = (q^t_i)_{i \in \mathcal{I}}$. The current on-hand inventory levels are denoted by $\mathrm{Inv}^t_{i,j}$ for each product $i \in \mathcal{I}$ at each location $j \in \mathcal{J}$. The core task is to create a dynamic fulfillment plan for each arriving order, assigning products to locations and carrier-services to minimize total expected costs while respecting operational constraints (see App.~\ref{app:nomen} for a full nomenclature).

For each order, the primary operational decision is the fulfillment plan $\mathbf{z}^{t} = (z^t_{i,j,k})_{i \in \mathcal{I},\, j \in \mathcal{J},\, k \in \mathcal{K}}$, where $z^t_{i,j,k}$ is the integer quantity of product $i$ sourced from location $j$ and shipped using carrier-service $k$. An auxiliary slack variable $\mathbf{u}^t = (u^t_i)_{i \in \mathcal{I}}$ represents unfulfilled demand for each product, ensuring feasibility when the requested quantity exceeds globally available inventory. The joint decision at epoch $t$ then belongs to the feasible set $\mathcal{Z}^t$ defined by fulfillment and inventory restrictions: 
\[
\mathcal{Z}^t
=
\left\{
(\mathbf{z}^t, \mathbf{u}^t) \in \mathbb{Z}^{|\mathcal{I}||\mathcal{J}||\mathcal{K}|}_{+} \times \mathbb{Z}^{|\mathcal{I}|}_{+} :
\sum_{j \in \mathcal{J}} \sum_{k \in \mathcal{K}} z^t_{i,j,k} + u^t_i = q^t_i, \ \forall i \in \mathcal{I}; \;
\sum_{k \in \mathcal{K}} z^t_{i,j,k} \le \mathrm{Inv}^t_{i,j}, \ \forall i \in \mathcal{I},\, j \in \mathcal{J}
\right\}.
\]

The decision-making process is subject to two primary sources of uncertainty. The first is delivery timeliness. Let $\tilde{\boldsymbol{\Delta}}^t = (\tilde{\Delta}^t_{j,k})_{j \in \mathcal{J},\, k \in \mathcal{K}}$ be a random vector, where $\tilde{\Delta}^t_{j,k}$ represents the deviation in days from the promised delivery date when an item is sourced from location $j$ and shipped using carrier-service $k$. A positive realization $\Delta^t_{j,k}$ indicates a late delivery, while a negative value indicates an early one. The realization of $\Delta^t_{j,k}$ for the chosen location-carrier pairs is observed only after the decision $\mathbf{z}^t$ is executed, analogous to a partial feedback setting \citep{ye2025contextual}. The second source of uncertainty is future demand, represented by a random vector $\tilde{\mathbf{D}}^t = (\tilde{D}^t_i)_{i \in \mathcal{I}}$ denoting the cumulative demand for each product from $t+1$ to $T$.

The per-epoch cost function $C(\mathbf{z}^t, \mathbf{u}^t, \tilde{\boldsymbol{\Delta}}^t)$ combines deterministic shipping costs, penalties for stochastic delivery deviations, and a unit stockout penalty $\rho$ applied to unfulfilled demand. To encourage consolidation, a discount $\beta \in (0,1)$ is applied to the shipping cost $c_{i,j,k}^{\mathrm{ship}}$ for any location-carrier pair $(j,k)$ that sources two or more units for the order. Asymmetric penalties, $\gamma^+ \gg \gamma^- > 0$, are applied for late ($\tilde{\Delta}_{j,k}^{\,t+} = \max\{0, \tilde{\Delta}^t_{j,k}\}$) and early ($\tilde{\Delta}_{j,k}^{\,t-} = \max\{0, -\tilde{\Delta}^t_{j,k}\}$) deliveries. The cost 
$
C(\mathbf{z}^t, \mathbf{u}^t, \tilde{\boldsymbol{\Delta}}^t)
$
is given by
\[
\sum_{j \in \mathcal{J}} \sum_{k \in \mathcal{K}}
\left[
\left(
\sum_{i \in \mathcal{I}} c_{i,j,k}^{\mathrm{ship}} z^t_{i,j,k}
\right)
\left(
1 - \beta \, \mathbf{1}\Bigl\{\sum_{i \in \mathcal{I}} z^t_{i,j,k} \ge 2\Bigr\}
\right)
+
\left( \gamma^+ \tilde{\Delta}_{j,k}^{\,t+} + \gamma^- \tilde{\Delta}_{j,k}^{\,t-} \right)
\sum_{i \in \mathcal{I}} z^t_{i,j,k}
\right]
+ \rho \sum_{i \in \mathcal{I}} u^t_i.
\]
As this cost function induces non-linear terms, in the computational implementation, the non-linear indicator function governing the consolidation discount is linearized using standard big-$M$ constraints and auxiliary binary variables.

At the start of epoch $t$, a contextual vector $\mathbf{X}^t \in \mathcal{X}$ is observed, containing features of the order, customer, and current network state. Because fulfillment decisions consume inventory and thereby affect future feasible actions, the problem is inherently sequential. The system state therefore evolves over time as inventory is depleted and new exogenous information arrives.

\subsection{Two-Stage Contextual Stochastic Model} \label{sec:2stg}
The sequential nature of online fulfillment gives rise to a multi-stage stochastic program that is typically intractable. A rolling-horizon framework with stage aggregation is therefore employed. At each epoch $t$, a two-stage contextual stochastic program is solved for the immediate decision $(\mathbf{z}^t, \mathbf{u}^t)$, collapsing future epochs $t+1, \dots, T$ into a single second-stage problem \citep{powell2022designing}. The first-stage decision is committed before any current-period or future uncertainty is realized. The second-stage variables $\mathbf{v}^t = (v^t_{i,j,k})_{i, j, k}$ and $\mathbf{w}^t = (w^t_i)_{i \in \mathcal{I}}$ represent the aggregated fulfillment plan and slack stockout for orders from $t+1$ to $T$. The problem solved at each epoch is
\begin{equation} \label{eq:full}
\min_{\;(\mathbf{z}^t, \mathbf{u}^t) \in \mathcal{Z}^t}
\;
\mathbb{E}_{\tilde{\boldsymbol{\Delta}}^t,\, \tilde{\mathbf{D}}^t \sim
\mathbb{P}(\cdot \mid \mathbf{X}^t)}
\Big[
C\!\big(\mathbf{z}^t, \mathbf{u}^t, \tilde{\boldsymbol{\Delta}}^t\big)
+
Q\!\big(\mathbf{z}^t, \tilde{\mathbf{D}}^t\big)
\Big],
\end{equation}
where $\mathbb{P}(\cdot \mid \mathbf{X}^t)$ represents the conditional distribution given the context $\mathbf{X}^t$ and $Q(\mathbf{z}^t, \tilde{\mathbf{D}}^t)$ is the second-stage cost function representing the expected future shipping and stockout costs:
\[
Q(\mathbf{z}^t, \tilde{\mathbf{D}}^t)
=
\min_{\mathbf{v}^{t} \ge 0,\, \mathbf{w}^t \ge 0}
\left\{
\sum_{i \in \mathcal{I}} \sum_{j \in \mathcal{J}} \sum_{k \in \mathcal{K}}
c^{\mathrm{ship}}_{i,j,k} \, v^t_{i,j,k} + \rho \sum_{i \in \mathcal{I}} w^t_i
\;\middle|\;
\begin{aligned}
&\sum_{j \in \mathcal{J}} \sum_{k \in \mathcal{K}} v^t_{i,j,k} + w^t_i
= \tilde{D}^t_i, && \forall i \in \mathcal{I}, \\
&\sum_{k \in \mathcal{K}} \bigl(z^t_{i,j,k} + v^t_{i,j,k}\bigr)
\le \mathrm{Inv}^t_{i,j}, && \forall i \in \mathcal{I},\, j \in \mathcal{J}.
\end{aligned}
\right\}.
\]
The first-stage decision $\mathbf{z}^t$ enters the second stage through the post-allocation inventory state, which constrains the feasible recourse and thereby determines $Q$. The second-stage variable $\mathbf{v}^t$ is declared continuous: with integer right-hand sides $\mathbf{Inv}^t$ and $\tilde{\mathbf{D}}^t$, the transportation-network constraint matrix is totally unimodular and an integer optimum is recovered without imposing integrality. The second-stage cost includes only base shipping and stockout costs on the forecast future demand; it omits future delivery-time penalties and consequently considers an approximation of the true future cost, since the aggregate delivery-time deviation across the remaining horizon cannot be forecast at decision time.

\subsection{Contextual Sample Average Approximation (C-SAA)} \label{sec:csaa}
To solve the stochastic program \eqref{eq:full}, the conditional distribution $\mathbb{P}(\cdot \mid \mathbf{X}^t)$ is approximated by a finite sampled scenario set $\tilde{\Omega}^t = \{\omega_1, \dots, \omega_E\}$ given the context $\mathbf{X}^t$, with each scenario $\omega$ providing a delivery-time realization $\boldsymbol{\Delta}^t_{\omega} = (\Delta^t_{j,k,\omega})_{j,k}$ and per-product cumulative demands $D^t_{i,\omega}$. The details of the scenario generation process and integration of the context information are discussed in Section \ref{sec:scenario_generation}. By introducing per-scenario second-stage variables $\mathbf{v}^t_{\omega} = (v^t_{i,j,k,\omega})$ and $\mathbf{w}^t_{\omega} = (w^t_{i,\omega})$, the C-SAA extensive form can be written as follows \citep{verweij2003sample, ye2025contextual}:
\begin{equation}
\label{eq:csaa_extensive}
\begin{aligned}
\min_{\substack{(\mathbf{z}^t, \mathbf{u}^t) \in \mathcal{Z}^t \\ \mathbf{v}^t_{\omega} \ge 0,\; \mathbf{w}^t_{\omega} \ge 0}} \quad 
& \frac{1}{E} \sum_{\omega \in \tilde{\Omega}^t} 
\Bigg(
C(\mathbf{z}^t, \mathbf{u}^t, \boldsymbol{\Delta}^t_{\omega})
+ \sum_{i \in \mathcal{I}} \sum_{j \in \mathcal{J}} \sum_{k \in \mathcal{K}} c^{\mathrm{ship}}_{i,j,k} \, v^t_{i,j,k,\omega}
+ \rho \sum_{i \in \mathcal{I}} w^t_{i,\omega}
\Bigg) \\
\text{s.t.} \quad 
& \sum_{j \in \mathcal{J}} \sum_{k \in \mathcal{K}} v^t_{i,j,k,\omega} + w^t_{i,\omega} = D^t_{i,\omega},
&& \forall i \in \mathcal{I},\, \omega \in \tilde{\Omega}^t, \\
& \sum_{k \in \mathcal{K}} \bigl(z^t_{i,j,k} + v^t_{i,j,k,\omega}\bigr) \le \mathrm{Inv}^{t}_{i,j},
&& \forall i \in \mathcal{I},\, j \in \mathcal{J},\, \omega \in \tilde{\Omega}^t .
\end{aligned}
\end{equation}

As a single finite-sample C-SAA instance is sensitive to scenario sampling variability, an iterative evaluation procedure is utilized to generate multiple candidate solutions under different scenario sets. The proposed procedure follows the standard multiple-replication strategy of the SAA algorithms \citep{verweij2003sample}, where $S$ candidate first-stage decisions are generated by solving independent instances of \eqref{eq:csaa_extensive} with scenario size $N_1$, then these solutions are evaluated on a larger common scenario set of size $N_2 \gg N_1$ to estimate their out-of-sample performances to select the best candidate solution. Details of this procedure is provided in Appendix~\ref{app:csaa_procedure}.

\section{A Scenario-Embedded Optimization Proxy for Sequential Contextual Stochastic Program}
\label{sec:proxy-framework}

Solving a solver-based scenario approximation of the two-stage contextual stochastic program could still be computationally prohibitive in real time during peak operating periods. This motivates the low-latency optimization proxy pipeline developed in this section to address these operational challenges. Section~\ref{sec:general_framework} states this design as a general proxy framework for sequential contextual stochastic programs. Section~\ref{sec:fulfillment_instantiation} instantiates the framework for the fulfillment setting of Section~\ref{sec:problem_formulation}.

\subsection{General Proxy Framework} \label{sec:general_framework}

Consider a sequential contextual stochastic program in which, at each epoch $t$, the decision maker observes a state or context $\mathbf{X}^t$, samples a scenario set $\mathbf{\Omega}^t$ from a fitted conditional forecaster, and chooses a first-stage decision in a feasible set $\mathcal{Z}^t$ before the next epoch. The computational bottleneck arises when a solver-based approximation, such as SAA or another scenario-based reformulation, can produce high-quality first-stage decisions offline but is too slow to run at every online epoch. The proxy framework replaces this online solve with a learned decision rule trained on offline solver labels.

The framework has four ingredients. First, an offline oracle maps historical contexts and scenario sets to first-stage action labels $\mathbf{a}^{t,*}$. Second, a scenario-embedded model $f_\theta(\mathbf{X}^t,\mathbf{\Omega}^t)$ maps the same information to decision scores or logits. Third, a decoder $D$ maps these scores into a feasible decision in $\mathcal{Z}^t$. Fourth, training combines label imitation with problem-aware regularization, such as a soft feasibility penalty and a scenario-based cost-alignment term. Online, the solver call is replaced by one forward pass through $f_\theta$ followed by the decoder $D$.

This framework is intended for latency-critical settings where high-quality offline labels are available and feasibility can be enforced by a decoder, projection, or efficient repair rule. The quality of the solutions obtained through this approach depends on the offline label generator, the fitted scenario model, and the expressiveness of the proxy class, which can be specialized depending on the problem setting considered. To this end, the fulfillment application below is one instantiation of the proposed proxy framework, where the offline oracle is finite-sample C-SAA, the generic action label $\mathbf{a}^{t,*}$ corresponds to the fulfillment decision $(\mathbf{z}^{t,*},\mathbf{u}^{t,*})$, the decoder is inventory-weighted feasible inference, and the scenario set contains demand and delivery-time samples.

\subsection{Fulfillment Instantiation} \label{sec:fulfillment_instantiation}

The remainder of this section instantiates the framework for omnichannel fulfillment. Figure~\ref{fig:fulfillment_proxy_overview} schematizes the pipeline: demand and delivery-time forecasters feed finite-sample C-SAA on historical orders offline, and the trained proxy serves each arriving order online.

\begin{figure}[!ht]
    \centering
    \includegraphics[width=0.82\linewidth]{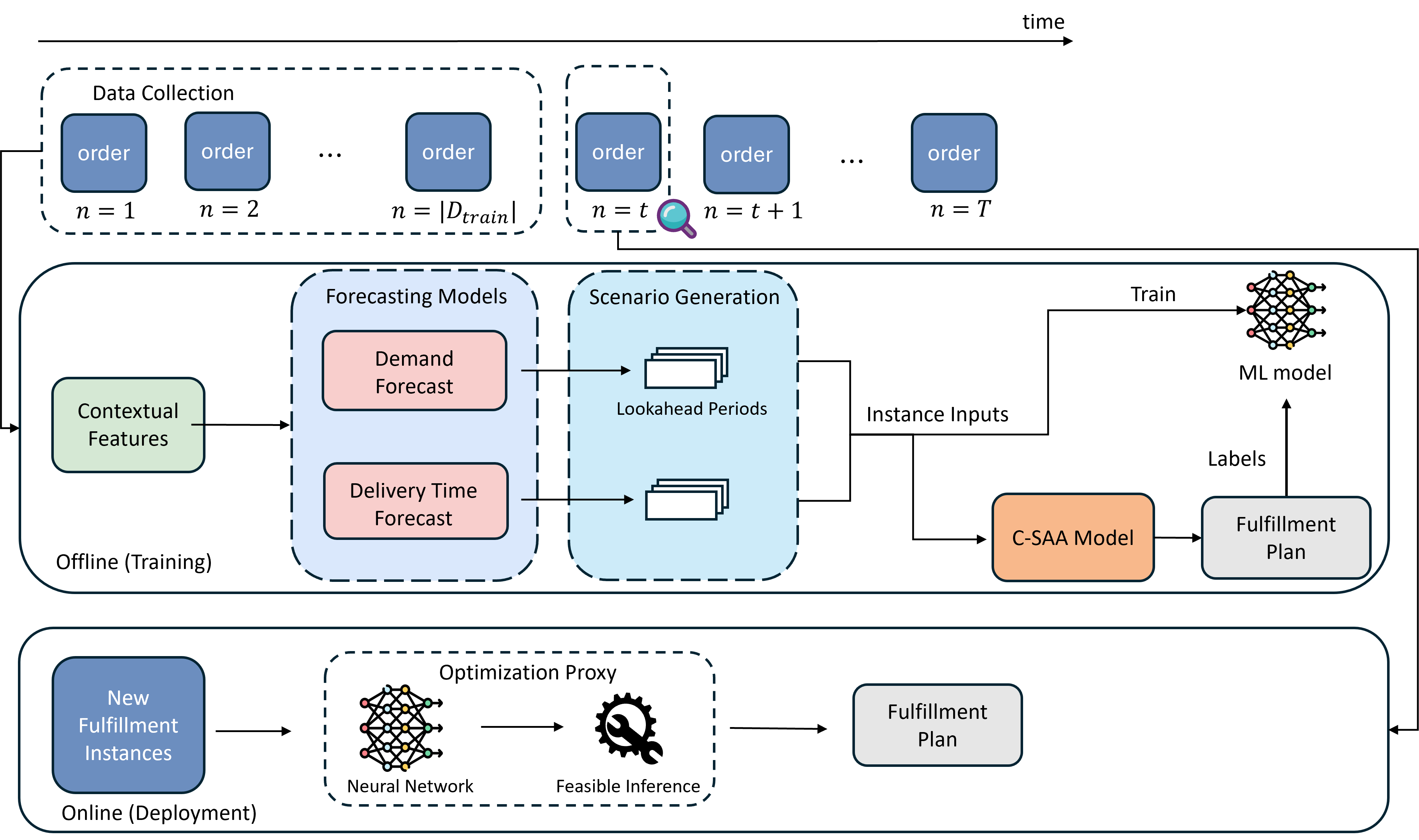}
    \caption{Schematic of the optimization proxy for omnichannel order fulfillment.}
    \label{fig:fulfillment_proxy_overview}
\end{figure}

\subsubsection{Contextual Inputs and Scenario Generation} \label{sec:scenario_generation}
The fulfillment context decomposes as
\(
\mathbf{X}^t = \bigl(\mathbf{X}_{\mathrm{ord}}^t,\; \{\mathbf{X}_{\mathrm{dc},j}^t\}_{j \in \mathcal{J}},\; \{\mathbf{X}_{\mathrm{cost},j,k}^t\}_{j \in \mathcal{J}, k \in \mathcal{K}}\bigr).
\)
The order-context vector $\mathbf{X}_{\mathrm{ord}}^t$ concatenates SKU and brand categorical embeddings with numeric order fields (number of SKUs in the order, total quantity, order type, bundle and coupon discount flags, gift-item flag, average discount rate), calendar features (e.g.\ hour and day of week), and customer demographics; the per-DC vector $\mathbf{X}_{\mathrm{dc},j}^t$ collects the on-hand inventory $\mathrm{Inv}_{j}^t$ and a DC categorical embedding for each candidate $j \in \mathcal{J}$; and the per-(DC, carrier) vector $\mathbf{X}_{\mathrm{cost},j,k}^t$ collects the base shipping cost $c^{\mathrm{ship}}_{j,k}$ and a carrier categorical embedding for each option $(j,k)$.

Two fitted forecasters approximate the conditional distribution $\mathbb{P}(\cdot \mid \mathbf{X}^t)$ through predicted quantiles, generating a paired scenario set $\mathbf{\Omega}^t = (\mathbf{\Omega}^t_{\mathrm{dem}},\,\mathbf{\Omega}^t_{\mathrm{time}})$ of $E$ Monte Carlo scenarios per random vector by inverse-transform sampling: for each $\omega \in \{1, \ldots, E\}$, a uniform draw $u_\omega \sim U(0,1)$ is mapped through the linearly interpolated predicted CDF to a realization. A multi-horizon quantile recurrent network (MQRNN) \citep{wen2017multi} consumes order-context features and historical demand to produce non-crossing per-period demand quantiles, which are inverse-sampled per period and summed across the horizon to yield $\mathbf{\Omega}^t_{\mathrm{dem}} = (\tilde{\Omega}^t_{i, \mathrm{dem}})_{i \in \mathcal{I}}$ with $\tilde{\Omega}^t_{i, \mathrm{dem}} = \{D_{i,1}^t, \ldots, D_{i,E}^t\}$. A multi-quantile multilayer perceptron (MQMLP) consumes order-context and per-(DC, carrier) cost features to produce delivery-time quantiles, sampled the same way to yield $\mathbf{\Omega}^t_{\mathrm{time}} = (\tilde{\Omega}^t_{j,k, \mathrm{time}})_{j \in \mathcal{J}, k \in \mathcal{K}}$ with $\tilde{\Omega}^t_{j,k, \mathrm{time}} = \{\Delta_{j,k,1}^{\,t}, \ldots, \Delta_{j,k,E}^{\,t}\}$. The online supplement reports the full feature lists, network architectures, training objectives, and held-out forecasting performance.

\subsubsection{Scenario-Embedded Hierarchical Proxy}
\label{sec:opt-proxy}

Because the constraints in $\mathcal{Z}^t$ are separable across products, a single shared network $f_\theta$ processes one product line at a time and the order-level decision concatenates per-product outputs. The consolidation discount couples the multiple product lines of the same order that ship together but is not separable; the per-product decomposition therefore sees this coupling only indirectly through the jointly optimized C-SAA labels. Each training instance $n \in \mathcal{N}$ is one historical (order, product-line) record, paired with its contextual features, sampled scenarios, and C-SAA label.

\paragraph{Architecture Walkthrough.} \label{sec:proxy_walkthrough}

Figure~\ref{fig:proxy_arch} schematizes the proxy. The upper panel shows the architecture: four input modules (global, scenario, DC, cost) produce embeddings that are concatenated per DC and passed to a DC head and a DC-conditioned carrier head, which together return $(\bm{\ell}^{\mathrm{dc}}, \bm{\ell}^{\mathrm{carrier}})$. The lower panels summarize the composite training loss applied at training time and the inventory-weighted decoder applied at inference (developed in Sections~\ref{sec:proxy_loss} and~\ref{sec:feasible_inference}). Figure~\ref{fig:proxy_forward} formalizes the forward pass of the optimization proxy. Each module $\phi$ and head $g$ is a multilayer perceptron (MLP) whose hidden layers apply a fully connected linear projection, batch normalization, dropout regularization, and a ReLU activation in sequence; hidden widths and layer counts are tuned and reported in Appendix~\ref{app:nn}.

\begin{figure}[!ht]
    \centering
    \includegraphics[width=0.88\linewidth]{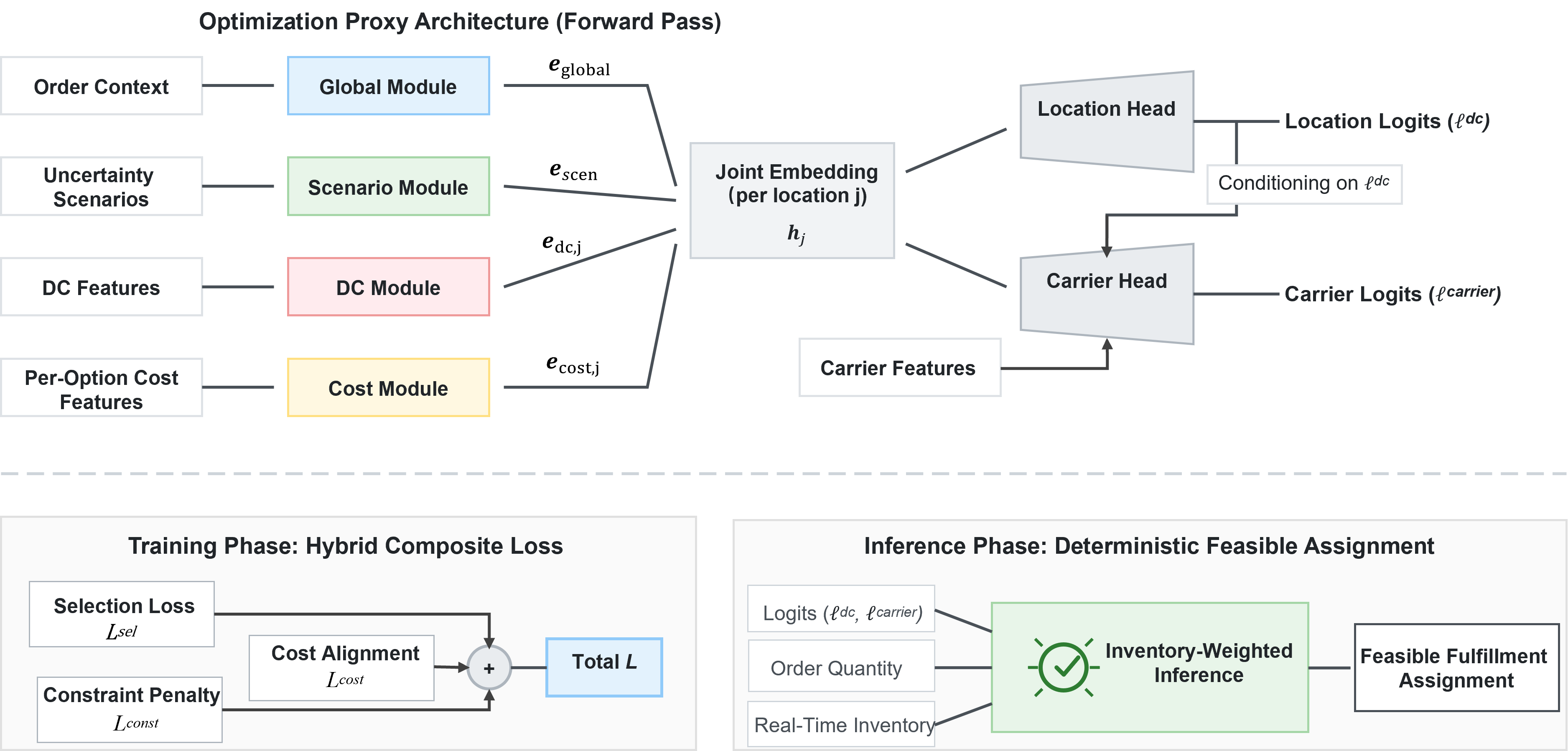}
    \caption{Schematic of optimization proxy architecture, training loss, and inference. 
    }
    \label{fig:proxy_arch}
\end{figure}

\begin{figure}[!ht]
\centering
\fbox{\parbox{\linewidth}{
$f_\theta\!\left(\mathbf{X}_{\mathrm{ord}},\; \{\mathbf{X}_{\mathrm{dc},j}\}_{j \in \mathcal{J}},\; \{\mathbf{X}_{\mathrm{cost},j,k}\}_{j \in \mathcal{J},\, k \in \mathcal{K}},\; \mathbf{\Omega}_{\mathrm{dem}},\; \mathbf{\Omega}_{\mathrm{time}}\right)$:
\begin{align*}
\mathbf{e}_{\mathrm{global}} &= \phi_{\mathrm{ord}}(\mathbf{X}_{\mathrm{ord}}) \\
\mathbf{m}^{\mathrm{dem}} &= \tfrac{1}{E}\textstyle\sum_{D_\omega \in \mathbf{\Omega}_{\mathrm{dem}}} \phi_{\mathrm{scen}}^{\mathrm{dem}}(D_\omega) \\
\mathbf{m}^{\mathrm{time}} &= \tfrac{1}{E}\textstyle\sum_{\boldsymbol{\Delta}_\omega \in \mathbf{\Omega}_{\mathrm{time}}} \phi_{\mathrm{scen}}^{\mathrm{time}}\!\left(\operatorname{vec}(\boldsymbol{\Delta}_\omega)\right) \\
\mathbf{e}_{\mathrm{scen}} &= \mathbf{m}^{\mathrm{dem}} + \mathbf{m}^{\mathrm{time}} \\
\mathbf{e}_{\mathrm{dc},j} &= \phi_{\mathrm{dc}}(\mathbf{X}_{\mathrm{dc},j}) & j \in \mathcal{J} \\
\mathbf{e}_{\mathrm{cost},j} &= \phi_{\mathrm{cost}}\!\left(\mathrm{stats}\!\left(\{c^{\mathrm{ship}}_{j,k}\}_{k \in \mathcal{K}},\; \{\Delta_{j,k,\omega}\}_{k \in \mathcal{K},\, \omega \in [E]}\right)\right) & j \in \mathcal{J} \\
\mathbf{h}_j &= [\mathbf{e}_{\mathrm{global}} \parallel \mathbf{e}_{\mathrm{scen}} \parallel \mathbf{e}_{\mathrm{dc},j} \parallel \mathbf{e}_{\mathrm{cost},j}] & j \in \mathcal{J} \\
\ell^{\mathrm{dc}}_j &= g_{\mathrm{dc}}(\mathbf{h}_j) & j \in \mathcal{J} \\
\ell^{\mathrm{carrier}}_{j,k} &= g_{\mathrm{carrier}}\!\left(\mathbf{h}_j,\; \ell^{\mathrm{dc}}_j,\; \mathbf{X}_{\mathrm{cost},j,k}\right) & j \in \mathcal{J},\, k \in \mathcal{K} \\
\text{return } & \bigl(\bm{\ell}^{\mathrm{dc}},\; \bm{\ell}^{\mathrm{carrier}}\bigr).
\end{align*}
}}
\caption{Forward pass of the optimization proxy $f_\theta$.}
\label{fig:proxy_forward}
\end{figure}

\paragraph{Global and DC modules.} A shared MLP $\phi_{\mathrm{ord}}$ encodes the order context $\mathbf{X}_{\mathrm{ord}}$ into a global embedding; a second shared MLP $\phi_{\mathrm{dc}}$ encodes each per-DC feature vector $\mathbf{X}_{\mathrm{dc},j}$ into a per-DC embedding, parameter-shared across candidates $j$ so the embedding size is independent of $|\mathcal{J}|$:
\begin{align*}
\mathbf{e}_{\mathrm{global}} &= \phi_{\mathrm{ord}}(\mathbf{X}_{\mathrm{ord}}), \quad
\mathbf{e}_{\mathrm{dc},j} = \phi_{\mathrm{dc}}(\mathbf{X}_{\mathrm{dc},j}), \qquad j \in \mathcal{J}.
\end{align*}

\paragraph{Scenario module.} The sampled scenario sets $\mathbf{\Omega}_{\mathrm{dem}}$ and $\mathbf{\Omega}_{\mathrm{time}}$ are unordered and may vary in size at deployment. For the current product line, identify $\mathbf{\Omega}_{\mathrm{dem}} = \{D_\omega\}_{\omega=1}^E$ as the scalar demand scenarios and $\mathbf{\Omega}_{\mathrm{time}} = \{\boldsymbol{\Delta}_\omega\}_{\omega=1}^E$ as the deviation matrices $\boldsymbol{\Delta}_\omega = (\Delta_{j,k,\omega})_{j \in \mathcal{J}, k \in \mathcal{K}}$ under a fixed $(j,k)$ ordering. Following the DeepSets template \citep{zaheer2017deep, patel2022neur2sp}, each scenario is encoded individually and the resulting embeddings are averaged. The encoder $\phi_{\mathrm{scen}}^{\mathrm{dem}}: \mathbb{R} \to \mathbb{R}^h$ (with $h$ the shared embedding dimension across the proxy's modules, a tuned hyperparameter reported in Appendix~\ref{app:nn}) maps $D_\omega$ to an embedding capturing where the scenario sits in the demand distribution. The encoder $\phi_{\mathrm{scen}}^{\mathrm{time}}: \mathbb{R}^{|\mathcal{J}||\mathcal{K}|} \to \mathbb{R}^h$ maps the flattened matrix $\operatorname{vec}(\boldsymbol{\Delta}_\omega)$ to an embedding capturing which DC--carrier options are jointly stressed in that sampled future. Averaging the per-scenario embeddings is the network's analogue of an expectation and matches the arithmetic-mean structure of the C-SAA objective:
\begin{align*}
\mathbf{m}^{\mathrm{dem}} &= \frac{1}{E}\sum_{D_\omega \in \mathbf{\Omega}_{\mathrm{dem}}} \phi_{\mathrm{scen}}^{\mathrm{dem}}(D_\omega), \quad
\mathbf{m}^{\mathrm{time}} = \frac{1}{E}\sum_{\boldsymbol{\Delta}_\omega \in \mathbf{\Omega}_{\mathrm{time}}} \phi_{\mathrm{scen}}^{\mathrm{time}}(\operatorname{vec}(\boldsymbol{\Delta}_\omega)), \\
\mathbf{e}_{\mathrm{scen}} &= \mathbf{m}^{\mathrm{dem}} + \mathbf{m}^{\mathrm{time}}.
\end{align*}
The two branches use independent weights at the same hidden width, so addition preserves disentangled demand and delivery-time signals without a learned fusion layer.

\paragraph{Cost module.} The candidate option set $\{(j,k)\}_{j,k}$ has $|\mathcal{J}|\,|\mathcal{K}|$ entries, so feeding all per-option features into the DC head directly would either inflate parameter count or require a second permutation-invariant pool over carriers. Instead, the proxy summarizes the per-DC cost and timeliness features through $\mathrm{stats}(\mathbf{c}_j, \boldsymbol{\Delta}_j) \in \mathbb{R}^8$, with $\mathbf{c}_j = (c^{\mathrm{ship}}_{j,k})_{k \in \mathcal{K}}$ and $\boldsymbol{\Delta}_j = (\Delta_{j,k,\omega})_{k \in \mathcal{K},\, \omega \in [E]}$. The first five outputs are carrier-axis statistics of $\mathbf{c}_j$ ($\overline{\mathbf{c}_j}, \min \mathbf{c}_j, \sigma(\mathbf{c}_j), p_{90}(\mathbf{c}_j), \mathrm{gap}_2(\mathbf{c}_j) = c^{(2)}_j - c^{(1)}_j$) corresponding to mean, minimum, standard deviation, 90th percentile, the gap between the two smallest entries, respectively, and the last three are (carrier, scenario)-axis statistics of $\boldsymbol{\Delta}_j$ ($\overline{\boldsymbol{\Delta}_j}, \sigma(\boldsymbol{\Delta}_j), p_{90}(\boldsymbol{\Delta}_j)$). The cost embedding is $\mathbf{e}_{\mathrm{cost},j} = \phi_{\mathrm{cost}}(\mathrm{stats}(\mathbf{c}_j, \boldsymbol{\Delta}_j))$. These dispersion statistics expose per-carrier variability to the DC head without committing to a specific carrier.

\paragraph{Joint embedding and hierarchical heads.} The four embeddings concatenate per DC, $\mathbf{h}_j = [\mathbf{e}_{\mathrm{global}} \parallel \mathbf{e}_{\mathrm{scen}} \parallel \mathbf{e}_{\mathrm{dc},j} \parallel \mathbf{e}_{\mathrm{cost},j}]$ for $j \in \mathcal{J}$. The DC head, MLP $g_{\mathrm{dc}}$, maps $\mathbf{h}_j$ to a scalar logit $\ell^{\mathrm{dc}}_j = g_{\mathrm{dc}}(\mathbf{h}_j)$. The carrier head $g_{\mathrm{carrier}}$ then conditions the carrier choice on the joint embedding, the raw DC logit, and the per-option features:
\begin{equation}
\label{eq:proxy_ell_carrier}
\ell^{\mathrm{carrier}}_{j,k} \;=\; g_{\mathrm{carrier}}\!\left(\mathbf{h}_j,\; \ell^{\mathrm{dc}}_j,\; \mathbf{X}_{\mathrm{cost},j,k}\right), \qquad j \in \mathcal{J},\, k \in \mathcal{K}.
\end{equation}
Two design choices matter in \eqref{eq:proxy_ell_carrier}. First, the carrier head receives the \emph{raw}, pre-softmax DC logit, so carrier selection can depend on the strength of the learned DC ranking rather than only on the relative ordering of DCs. Second, the two heads imply a hierarchical factorization $p(j) \cdot p(k \mid j)$. The DC-head softmax realizes $p(j)$ over distribution centers, and the carrier-head softmax realizes $p(k \mid j)$ over each DC's eligible carrier services. This factorization matches the structure of the fulfillment decision: an order line first selects a DC, then a carrier-service from that DC's eligible set. It reduces the complexity of learning the full joint $|\mathcal{J}|\,|\mathcal{K}|$ assignment.

\subsubsection{Composite Training Loss} \label{sec:proxy_loss}

Figure~\ref{fig:proxy_training} states the training optimization. It takes the $|\mathcal{N}|$ training instances as input and minimizes the average per-instance composite loss $\mathcal{L}^n$, which combines three terms targeting C-SAA-label imitation, operational feasibility, and self-supervised cost alignment.

The offline C-SAA oracle returns a quantity decision $(\mathbf{z}^{n,*},u^{n,*})$ for the product line, which may distribute the requested quantity across multiple $(j,k)$ options. The learning task is framed as predicting a single \emph{primary} option per product line rather than the full multi-DC split, which encourages shipping consolidation; if a shortage forces a split, the inventory-weighted decoder of Section~\ref{sec:feasible_inference} recovers it at inference. The primary location and carrier-service labels are extracted as the most-served option of the C-SAA decision:
\begin{equation*}
j^{n,*} \in \arg\max_{j\in\mathcal{J}}\sum_{k\in\mathcal{K}} z^{n,*}_{j,k},
\qquad
k^{n,*} \in \arg\max_{k\in\mathcal{K}_{j^{n,*}}^n} z^{n,*}_{j^{n,*},k},
\end{equation*}
with deterministic tie breaking.

\begin{figure}[!ht]
\centering
\fbox{\parbox{\linewidth}{
\begin{align*}
\theta^* \;\in\; \arg\min_{\theta} \quad & \frac{1}{|\mathcal{N}|}\sum_{n \in \mathcal{N}} \mathcal{L}^n(\theta) \\
\text{s.t.} \quad & \bigl(\bm{\ell}^{n,\mathrm{dc}},\; \bm{\ell}^{n,\mathrm{carrier}}\bigr) \;=\; f_\theta(\mathbf{X}^n,\mathbf{\Omega}^n) & \forall n \in \mathcal{N} \\
& \mathbf{p}^{n,\mathrm{dc}} \;=\; \mathrm{softmax}\!\bigl(\bm{\ell}^{n,\mathrm{dc}}\bigr) & \forall n \in \mathcal{N} \\
& \mathbf{p}^{n,\mathrm{carrier}} \;=\; \mathrm{softmax}_k\!\bigl(\bm{\ell}^{n,\mathrm{carrier}}\bigr) & \forall n \in \mathcal{N} \\
& \mathbf{a}^n \;=\; \mathrm{ST\text{-}GumbelSoftmax}(\bm{\ell}^{n,\mathrm{dc}},\, \tau) & \forall n \in \mathcal{N} \\
& \mathcal{L}^n(\theta) \;=\; \mathcal{L}^n_{\mathrm{sel}}\!\bigl(\bm{\ell}^{n,\mathrm{dc}},\, \bm{\ell}^{n,\mathrm{carrier}};\, j^{n,*},\, k^{n,*}\bigr) \\
& \qquad\qquad +\; \mathcal{L}^n_{\mathrm{const}}\!\bigl(\mathbf{a}^n;\, q^n,\, \mathbf{Inv}^n\bigr) \\
& \qquad\qquad +\; \mathcal{L}^n_{\mathrm{cost}}\!\bigl(\mathbf{p}^{n,\mathrm{dc}},\, \mathbf{p}^{n,\mathrm{carrier}};\, \mathbf{\Omega}^n_{\mathrm{cost}}\bigr).
\end{align*}
}}
\caption{Optimization model for training the proxy $f_\theta$.}
\label{fig:proxy_training}
\end{figure}

\paragraph{Selection Loss.} Let $\mathbf{p}^{n,\mathrm{dc}} = \mathrm{softmax}(\bm{\ell}^{n,\mathrm{dc}})$ and $\mathbf{p}^{n,\mathrm{carrier}} = \mathrm{softmax}_k(\bm{\ell}^{n,\mathrm{carrier}})$ denote the head probabilities, with the carrier softmax taken along the carrier-service axis at each fixed DC. The selection loss $\mathcal{L}_{\mathrm{sel}}$ supervises both heads against the primary labels $(j^{n,*}, k^{n,*})$ via the negative log-likelihood of the labeled pair under the hierarchical predicted distribution (equivalently, the categorical cross-entropy on the unnormalized logits), with separate weights on the DC and conditional-carrier terms:
\begin{equation*}
\mathcal{L}^n_{\mathrm{sel}} \;=\; -\lambda_{\mathrm{sel}} \Bigl(\log p^{n,\mathrm{dc}}_{j^{n,*}} \;+\; \lambda_{\mathrm{carrier}}\, \log p^{n,\mathrm{carrier}}_{j^{n,*},\, k^{n,*}}\Bigr).
\end{equation*}
Here $\lambda_{\mathrm{sel}}$ scales the overall selection weight and $\lambda_{\mathrm{carrier}}$ scales the relative importance of the conditional carrier signal.

\paragraph{Constraint Loss.} The one-hot DC assignment $\mathbf{a}^n \in \{0,1\}^{|\mathcal{J}|}$ is drawn from the DC logits via the Straight-Through Gumbel-Softmax (ST-GumbelSoftmax) estimator of \citet{jang2017categorical}: $\mathbf{a}^n$ is the one-hot encoding of $\arg\max_j (\ell^{n,\mathrm{dc}}_j + g_j)$ with independent and identically distributed (i.i.d.) standard Gumbel noise $g_j$, but its gradient is computed as if $\mathbf{a}^n$ equaled the temperature-$\tau$ softmax $\mathrm{softmax}(\bm{\ell}^{n,\mathrm{dc}}/\tau)$, so a discrete DC choice can pass to the loss while gradients still flow back through $g_{\mathrm{dc}}$. The discrete sample is essential here: evaluating the inventory cap on the soft mass $\sum_j p^{n,\mathrm{dc}}_j \min\{q^n, \mathrm{Inv}^n_j\}$ instead would let a hedge between an infeasible primary and a feasible secondary DC escape penalty.

Two feasibility requirements in $\mathcal{Z}^t$ bind on the DC selected by $\mathbf{a}^n$: the per-DC inventory cap ($\sum_k z_{j,k} \le \mathrm{Inv}^n_j$) and the demand-satisfaction requirement that the requested quantity $q^n$ be served. The constraint penalty $\mathcal{L}_{\mathrm{const}}$ encodes both as the demand shortfall after each DC's allocation is capped by its on-hand inventory:
\begin{equation*}
\mathcal{L}^n_{\mathrm{const}} \;=\; \lambda_{\mathrm{const}} \left[\, q^n \;-\; \sum_{j \in \mathcal{J}} \min\!\bigl\{q^n a^n_j,\; \mathrm{Inv}^n_{j}\bigr\} \,\right]^+,
\end{equation*}
where $[x]^+ = \max(0, x)$ and $\lambda_{\mathrm{const}}$ scales the penalty. Because $\mathbf{a}^n$ is one-hot, the served quantity reduces to $\min\{q^n, \mathrm{Inv}^n_{j^*}\}$ at the chosen primary DC $j^*$, so the penalty fires precisely when that DC's inventory is insufficient.

\paragraph{Cost-Alignment Loss.} The cost-alignment penalty $\mathcal{L}_{\mathrm{cost}}$ is a self-supervised regularizer that pulls probability mass toward low-cost (DC, carrier-service) pairs. It operates on the head probabilities $(\mathbf{p}^{n,\mathrm{dc}}, \mathbf{p}^{n,\mathrm{carrier}})$ defined above, so the gradient flows through a smooth surface. Letting $c^n_{j,k,\omega}$ be the realized \emph{immediate} fulfillment cost for option $(j,k)$ under scenario $\omega$ (shipping plus delivery-time penalty, no second-stage recourse) drawn from $\mathbf{\Omega}^n_{\mathrm{cost}}$, the term computes the expected immediate cost under the predicted joint $p^{n,\mathrm{dc}}_j p^{n,\mathrm{carrier}}_{j,k}$:
\begin{equation*}
\mathcal{L}^n_{\mathrm{cost}} \;=\; \lambda_{\mathrm{cost}}\, \frac{1}{E} \sum_{\omega=1}^E \sum_{j \in \mathcal{J}} \sum_{k \in \mathcal{K}} p^{n,\mathrm{dc}}_j\, p^{n,\mathrm{carrier}}_{j,k}\, c^n_{j,k,\omega}.
\end{equation*}

\subsubsection{Feasible Inference via Inventory-Weighted Decoding}
\label{sec:feasible_inference}

At an online epoch $t$, the trained proxy first maps the current context and sampled scenarios to the same hierarchical logits used during training,
\begin{equation*}
(\bm{\ell}^{\mathrm{dc}}, \bm{\ell}^{\mathrm{carrier}})
= f_\theta\!\left(\mathbf{X}_{\mathrm{ord}}^t,\{\mathbf{X}_{\mathrm{dc},j}^t\}_{j\in\mathcal{J}},
\{\mathbf{X}_{\mathrm{cost},j,k}^t\}_{j,k},\mathbf{\Omega}^t_{\mathrm{dem}},\mathbf{\Omega}^t_{\mathrm{time}}\right),
\end{equation*}
and converts them to probabilities $\mathbf{p}^{\mathrm{dc}} = \mathrm{softmax}(\bm{\ell}^{\mathrm{dc}})$ and $\mathbf{p}^{\mathrm{carrier}} = \mathrm{softmax}_k(\bm{\ell}^{\mathrm{carrier}})$. These probabilities are not used directly as fractional decisions. Instead, the decoder reweights each DC's predicted probability by the fractional coverage that the DC can offer for the current line: for a product line with requested quantity $q^t$,
\begin{equation*}
r_j^t=\min\{\mathrm{Inv}_j^t/q^t,1\},\qquad
s_j^t=p_j^{\mathrm{dc}} r_j^t,
\end{equation*}
where $r_j^t \in [0,1]$ leaves a fully-covering DC undisturbed and discounts a partially-covering one in proportion to its coverage. The reweighting approximates the second-stage recourse cost $Q$ that the training loss does not directly model: depleting a DC whose stock barely covers the current line raises the likelihood of a costly split or stockout on a future order that the same DC could have served, so steering volume toward DCs with healthier inventory buffers preserves serviceability over the horizon. The score $s_j^t$ also softens the single-DC preference learned through $\mathcal{L}_{\mathrm{const}}$, allowing multi-DC splits when shortages require them.

The decoder then sorts eligible DCs by $s_j^t$ and scans them greedily. For each visited DC, it chooses the highest-probability eligible carrier service and assigns as much residual demand as inventory permits, where $\mathcal{K}_j^t$ denotes the set of carrier services eligible at DC $j$ at epoch $t$:
\begin{equation*}
k^*(j)\in\arg\max_{k\in\mathcal{K}_j^t}p_{j,k}^{\mathrm{carrier}},\qquad
x_j^t=\min\{q_{\mathrm{rem}}^t,\mathrm{Inv}_j^t\},\qquad
z_{j,k^*(j)}^{\mathrm{proxy},t}=x_j^t.
\end{equation*}
The residual $q_{\mathrm{rem}}^t$ (the quantity of the line not yet assigned, initialized to the requested quantity $q^t$) is updated after each assignment, and any remaining quantity becomes unmet demand $u^{\mathrm{proxy},t}$. Thus the learned probabilities determine the ordering and carrier choice, while the decoder enforces the demand-allocation, inventory, and eligibility constraints in $\mathcal{Z}^t_i$ corresponding to the product-line $i$ restriction of $\mathcal{Z}^t$, which can be defined as $\mathcal{Z}^t_i
=\{
(\mathbf{z}^t, \mathbf{u}^t) \in \mathbb{Z}^{|\mathcal{J}||\mathcal{K}|}_{+} \times \mathbb{Z}_{+}:
\sum_{j \in \mathcal{J}} \sum_{k \in \mathcal{K}} z^t_{j,k} + u^t = q^t; \;
\sum_{k \in \mathcal{K}} z^t_{j,k} \le \mathrm{Inv}^t_{j},\, j \in \mathcal{J}\}$. Algorithm~\ref{alg:proxy_decoder} summarizes the end-to-end online inference procedure for one product line.

\begin{algorithm}[!ht]
\caption{End-to-end online inference for one product line $i$: $f_\theta$ maps context and scenarios to logits, and the inventory-weighted decoder returns a feasible decision in $\mathcal{Z}^t_i$ (the product-line restriction of $\mathcal{Z}^t$).}
\label{alg:proxy_decoder}
\begin{algorithmic}[1]
\Require Order context $\mathbf{X}_{\mathrm{ord}}^t$, per-DC features $\{\mathbf{X}_{\mathrm{dc},j}^t\}_{j \in \mathcal{J}}$, per-option features $\{\mathbf{X}_{\mathrm{cost},j,k}^t\}_{j \in \mathcal{J},\, k \in \mathcal{K}}$, sampled scenarios $\mathbf{\Omega}^t_{\mathrm{dem}}, \mathbf{\Omega}^t_{\mathrm{time}}$, requested quantity $q^t \in \mathbb{Z}_+$, inventories $\{\mathrm{Inv}_j^t \in \mathbb{Z}_+\}_{j \in \mathcal{J}}$, eligibility sets $\{\mathcal{K}_j^t \subseteq \mathcal{K}\}_{j \in \mathcal{J}}$.
\Ensure  Feasible decision $(\mathbf{z}^{\mathrm{proxy},t}, u^{\mathrm{proxy},t}) \in \mathcal{Z}^t_i$.
\State $(\bm{\ell}^{\mathrm{dc}}, \bm{\ell}^{\mathrm{carrier}}) \gets f_\theta\bigl(\mathbf{X}_{\mathrm{ord}}^t, \{\mathbf{X}_{\mathrm{dc},j}^t\}, \{\mathbf{X}_{\mathrm{cost},j,k}^t\}, \mathbf{\Omega}^t_{\mathrm{dem}}, \mathbf{\Omega}^t_{\mathrm{time}}\bigr)$ \Comment{forward pass}
\State $\mathbf{p}^{\mathrm{dc}} \gets \mathrm{softmax}(\bm{\ell}^{\mathrm{dc}})$;\quad $\mathbf{p}^{\mathrm{carrier}} \gets \mathrm{softmax}_k(\bm{\ell}^{\mathrm{carrier}})$
\State For each $j \in \mathcal{J}$: $r_j^t \gets \min\{\mathrm{Inv}_j^t/q^t,\, 1\}$;\quad $s_j^t \gets p_j^{\mathrm{dc}}\, r_j^t$ \Comment{inventory-weighted score}
\State $\mathcal{J}^{\mathrm{elig}} \gets \{j \in \mathcal{J} : \mathrm{Inv}_j^t > 0,\, \mathcal{K}_j^t \neq \emptyset\}$, sorted by $s_j^t$ in decreasing order
\State $q_{\mathrm{rem}} \gets q^t$;\quad $z_{j,k}^{\mathrm{proxy},t} \gets 0$ for all $(j,k) \in \mathcal{J} \times \mathcal{K}$
\For{$j$ in $\mathcal{J}^{\mathrm{elig}}$ \textbf{while} $q_{\mathrm{rem}} > 0$}
  \State $x_j \gets \min\{q_{\mathrm{rem}},\, \mathrm{Inv}_j^t\}$
  \State $k^*(j) \gets \arg\max_{k \in \mathcal{K}_j^t} p_{j,k}^{\mathrm{carrier}}$
  \State $z_{j,k^*(j)}^{\mathrm{proxy},t} \gets x_j$;\quad $q_{\mathrm{rem}} \gets q_{\mathrm{rem}} - x_j$
\EndFor
\State $u^{\mathrm{proxy},t} \gets q_{\mathrm{rem}}$
\State \Return $(\mathbf{z}^{\mathrm{proxy},t}, u^{\mathrm{proxy},t})$
\end{algorithmic}
\end{algorithm}

Proposition~\ref{prop:proxy_feasibility} establishes feasibility, runtime, and policy-value estimation properties of the proposed framework.

\begin{proposition}[Feasibility, Runtime, and Policy-Value Bound]
\label{prop:proxy_feasibility}
For each product line, Algorithm~\ref{alg:proxy_decoder} returns a feasible decision in $\mathcal{Z}^t_i$ and runs in $O(|\mathcal{J}|\log|\mathcal{J}| + \sum_j |\mathcal{K}_j^t|)$ time, or $O(|\mathcal{J}|\log|\mathcal{J}| + |\mathcal{J}||\mathcal{K}|)$ when every carrier is eligible at every DC. Since $\mathcal{Z}^t=\prod_{i\in\mathcal{I}}\mathcal{Z}^t_i$, concatenating the per-product decisions gives $(\mathbf{z}^{\mathrm{proxy},t},\mathbf{u}^{\mathrm{proxy},t}) \in \mathcal{Z}^t$. Define
\(
f_{\mathrm{proxy}}^t := \mathbb{E}\!\left[ C(\mathbf{z}^{\mathrm{proxy},t},\mathbf{u}^{\mathrm{proxy},t},\tilde{\boldsymbol{\Delta}}^t) + Q(\mathbf{z}^{\mathrm{proxy},t},\tilde{\mathbf{D}}^t) \,\middle|\, \mathbf{X}^t \right],
\)
where the expectation is taken jointly over $(\tilde{\boldsymbol{\Delta}}^t, \tilde{\mathbf{D}}^t)$.
Then $f_{\mathrm{proxy}}^t$ upper-bounds the optimal objective value of the two-stage contextual stochastic program~\eqref{eq:full} at epoch $t$.
\end{proposition}

\begin{proof}{Proof}
Since the constraints in $\mathcal{Z}^t$ are separable across products with  $\mathcal{Z}^t = \prod_{i \in \mathcal{I}} \mathcal{Z}^t_i$,
and the order-level decision concatenates per-product outputs of Algorithm~\ref{alg:proxy_decoder}, it suffices to verify $(\mathbf{z}^{\mathrm{proxy},t}_i, u^{\mathrm{proxy},t}_i) \in \mathcal{Z}^t_i$ for the product $i$. Each loop iteration assigns an integer $x_j = \min\{q_{\mathrm{rem}}, \mathrm{Inv}_j^t\} \le \mathrm{Inv}_j^t$ to a single carrier at a previously unvisited $j$, so $\sum_k z^{\mathrm{proxy},t}_{j,k} \le \mathrm{Inv}_j^t$, and the loop invariant $q^t = q_{\mathrm{rem}} + \sum_j x_j$ gives $\sum_{j,k} z^{\mathrm{proxy},t}_{j,k} + u^{\mathrm{proxy},t} = q^t$ on termination. Hence $(\mathbf{z}^{\mathrm{proxy},t},\mathbf{u}^{\mathrm{proxy},t}) \in \mathcal{Z}^t$. Then, by utilizing this feasible solution, an upper bound can be constructed as $f_{\mathrm{proxy}}^t = \mathbb{E}[\,C(\mathbf{z}^{\mathrm{proxy},t},\mathbf{u}^{\mathrm{proxy},t},\tilde{\boldsymbol{\Delta}}^t) + Q(\mathbf{z}^{\mathrm{proxy},t},\tilde{\mathbf{D}}^t) \mid \mathbf{X}^t\,] \ge \min_{(\mathbf{z},\mathbf{u}) \in \mathcal{Z}^t} \mathbb{E}[\,C(\mathbf{z},\mathbf{u},\tilde{\boldsymbol{\Delta}}^t) + Q(\mathbf{z},\tilde{\mathbf{D}}^t)\mid \mathbf{X}^t\,]$.

For runtime, sorting eligible DCs by $s_j^t$ costs $O(|\mathcal{J}|\log|\mathcal{J}|)$. The allocation scan over sorted DCs costs $O(|\mathcal{J}|)$, and selecting $k^*(j)$ scans the eligible carrier sets at total cost $\sum_j |\mathcal{K}_j^t|$. Thus, the decoder costs $O(|\mathcal{J}|\log|\mathcal{J}| + \sum_j |\mathcal{K}_j^t|)$, showing the desired result. When all carriers are eligible at all DCs, $\sum_j |\mathcal{K}_j^t|=|\mathcal{J}||\mathcal{K}|$.
\end{proof}

To evaluate the policy obtained in practice, an independent out-of-sample scenario set can be utilized. The following remark clarifies the statistical interpretation of this evaluation procedure.

\begin{remark}
\label{remark:ProxyEstimators}
To estimate the expected performance of the policy $(\mathbf{z}^{\mathrm{proxy},t},\mathbf{u}^{\mathrm{proxy},t})$ returned by Algorithm~\ref{alg:proxy_decoder} at epoch $t$, consider an independent scenario set $\Omega_{N_2}^t = \{\omega_\ell\}_{\ell=1}^{N_2}$ where  the scenarios are drawn i.i.d. from an evaluation distribution $\mathbb{Q}(\cdot\mid X_t)$ and each scenario $\omega\in\Omega_{N_2}^t$ specifies a joint realization $(\boldsymbol{\Delta}_\omega^t, \mathbf{D}_\omega^t)$. Let $Y_\omega^t := C(\mathbf{z}^{\mathrm{proxy},t},\mathbf{u}^{\mathrm{proxy},t},\boldsymbol{\Delta}_\omega^t) + Q(\mathbf{z}^{\mathrm{proxy},t},\mathbf{D}_\omega^t)$.
Assuming $\mathbb{E}_{\mathbb{Q}}[\,(Y_\omega^t)^2 \mid \mathbf{X}^t\,] < \infty$, the sample-average estimator $\bar f_t^{\mathrm{proxy}}=\frac{1}{N_2}\sum_{\omega\in\Omega_{N_2}^t}Y_\omega^t$ is unbiased for the conditional policy value $\mathbb{E}_{\mathbb{Q}}[\,Y_\omega^t \mid \mathbf{X}^t\,]$ and converges almost surely to this value as $N_2 \to \infty$. For $N_2 \ge 2$, $(s_{\mathrm{proxy}}^t)^2 = \frac{1}{N_2-1}\sum_{\omega\in\Omega_{N_2}^t}(Y_\omega^t - \bar f_t^{\mathrm{proxy}})^2$ is unbiased and consistent for $\mathrm{Var}_{\mathbb{Q}}(Y_\omega^t \mid \mathbf{X}^t)$.
\end{remark}

\subsection{Training Data Collection and Model Selection}
This Section describes the experimental protocol used to generate proxy-training labels and evaluate out-of-sample policy performance. The training horizon spans three consecutive weeks, partitioned chronologically: the initial two weeks form the forecast training split $\mathcal{D}_{\mathrm{fcast}}$, the third week constitutes the proxy training split $\mathcal{D}_{\mathrm{proxy}}$, and a held-out fourth week is used for online evaluation. Forecasters are trained on $\mathcal{D}_{\mathrm{fcast}}$ alone, with no exposure to $\mathcal{D}_{\mathrm{proxy}}$ or the evaluation week to prevent leakage into labels or test; they then generate context-conditional scenarios for each order in $\mathcal{D}_{\mathrm{proxy}}$, and C-SAA solves the resulting instances against the inventory state at each order's arrival epoch to produce the C-SAA-optimal label $\mathbf{z}^{n,*}$ (optimal over the finite scenario sample, not the true conditional distribution). The same scenario tensors feed both C-SAA (for labels) and the proxy (as inputs), so the proxy is trained on the uncertainty distribution that produced its labels.

The proxy is then trained on $\mathcal{D}_{\mathrm{proxy}}$ with hierarchical location-carrier labels $(j^{n,*}, k^{n,*})$ extracted from $\mathbf{z}^{n,*}$. Before the final evaluation, the forecasters are refitted on $\mathcal{D}_{\mathrm{train}} = \mathcal{D}_{\mathrm{fcast}} \cup \mathcal{D}_{\mathrm{proxy}}$ for deployment fidelity; the proxy itself is not retrained. A temporal subset of $\mathcal{D}_{\mathrm{proxy}}$ is held out as a validation set, and the proxy checkpoint is selected by Hit@5, defined as the rate at which the C-SAA-chosen (DC, carrier) pair appears in the proxy's top-five predicted options. Top-five is used because multiple near-optimal pairs typically exist per instance, and at deployment the inventory-weighted decoder reshuffles top-ranked options against real-time inventory.

\section{Computational Study} \label{sec:computational_study}
The proxy is evaluated in a simulated online environment built from JD.com transaction records \citep{shen2024jd}, with independent daily instances spanning the 06:00--18:00 peak operational window. The planning horizon $\mathcal{T}$ corresponds to this window, and second-stage demand $\tilde{\mathbf{D}}^t$ aggregates remaining intra-day arrivals from the demand forecaster's 24-hour rolling output. Inventory and queue states are reinitialized each day from historical demand parameters with no intra-day replenishment. The primary metric is cumulative realized objective value, comprising base shipping cost, delivery-time penalties, multi-unit consolidation discounts, and stockout penalties. Late delivery percentage and cumulative lateness are reported as service-level metrics.

Because the public JD.com release is dominated by a single carrier and lacks exact facility coordinates, the dataset is augmented with a synthetic multi-carrier-service layer and synthetic geocoordinates, calibrated against a separate proprietary multi-carrier dataset. Preprocessing, calibration, and instance construction are detailed in Appendices~\ref{app:instance} and~\ref{app:prep}. The first three weeks supply training data and the final week is held out for online evaluation. A counterfactual simulator generates calibrated delivery-time realizations for unobserved fulfillment decisions \citep{ye2025contextual}, and all reported metrics average 50 independent simulation runs.

Evaluation proceeds order by order in rolling-horizon fashion, with the inventory state updated after each decision. The C-SAA solver uses Gurobi 13.0.1 with four threads per instance, $S=10$ candidate solutions at $N_1=50$ scenarios each, common-set evaluation on $N_2=500$ scenarios, and a 0.01\% relative MIP gap. Experiments ran on a shared HPC cluster (RHEL 9.6, dual Intel Xeon Gold 6226 at 2.7\,GHz, NVIDIA RTX 6000 GPUs).

\subsection{Comparative Analysis} \label{sec:comparative}
The proxy and the online C-SAA solver are compared against five baselines covering heuristic, point-forecast, and benchmark fulfillment approaches. The first three baselines are extended to the sequential setting from \citet{ye2025contextual}.
\begin{itemize}[leftmargin=*,topsep=0pt,itemsep=0pt,parsep=0pt,partopsep=0pt]
    \item \textbf{Greedy:} minimizes the item-specific base shipping cost $c^{\mathrm{ship}}_{i,j,k}$ for each product $i$ and ignores stochastic delivery-time deviations. It mirrors the common practice of always sourcing from the cheapest feasible option.
    \item \textbf{Predict-Then-Optimize (PTO):} substitutes $\mathbb{P}(\cdot \mid \mathbf{X}^t)$ with its conditional mean and solves the deterministic equivalent. It is a forecast-driven policy that acts on expected values and ignores the spread of uncertainty.
    \item \textbf{Empirical-SAA:} samples $\mathbf{\Omega}^t$ uniformly from the historical unconditional empirical distributions, ignoring contextual signals. It accounts for uncertainty but not for the context of the specific order.
    \item \textbf{Deterministic LP (DTLP)} \citep{acimovic2015making}: solves a linear program on a demand forecast to obtain dual prices that value the future opportunity cost of consuming inventory, then fulfills each arriving order against those prices.
    \item \textbf{Primal-Dual} \citep{andrews2019primal}: maintains dual prices on inventory and updates them online as orders arrive, without relying on a demand forecast.
\end{itemize}

Table~\ref{tab:sim_summary_test} reports aggregate out-of-sample performance over the test week: the total realized objective (base shipping cost, delivery-time penalties, consolidation discounts, and stockout penalties), the average late-delivery rate and cumulative lateness, the per-order runtime, and the $q_{50}$--$q_{95}$ quantiles of each. In the calibrated simulation environment, the proxy attains a total objective of $829{,}805$, improving on the online finite-sample C-SAA reference ($857{,}892$) by 3.3\%. The average late delivery rate falls from 5.71\% to 4.93\%, and the daily gap is stable. The proxy's edge over its own labels reflects scenario-noise averaging. The C-SAA reference at $(S,N_1,N_2)=(10,50,500)$ is itself exposed to scenario generation error, finite-$N_1$ sampling variance, and two-stage aggregation bias, while the proxy averages over the full corpus of these noisy labels with $\mathcal{L}_{\mathrm{cost}}$ as a structural regularizer. The proxy improves on C-SAA in aggregate but can lag where C-SAA's finite sample lands near the true optimum, as on the largest tier in Table~\ref{tab:runtime_gap_size_test}. A higher-budget C-SAA may narrow the aggregate gap. Deterministic and point-prediction baselines lag substantially: Greedy reaches $928{,}970$ with a 9.80\% late rate, and PTO and Empirical-SAA both exceed $1{,}000{,}000$ in objective and 11\% in late rate, confirming the value of scenario-aware contextual information \citep{ye2025contextual}.

The proxy executes each decision in 0.006 seconds, an acceleration above $2{,}800\times$ over C-SAA's 17.084 seconds. It is also faster than the simpler Greedy (0.104s) and PTO (0.223s) baselines, fitting within the sub-second latency budgets of high-throughput fulfillment systems.

Relative to DTLP, the strongest classical baseline on objective, the proxy lowers total cost by $10.7\%$ ($829{,}805$ versus $928{,}889$) and roughly halves the late-delivery rate ($4.93\%$ versus $9.77\%$); the cost gap widens to $48.7\%$ against the primal-dual algorithm. Both run within the same sub-second budget as the proxy, yet neither matches it on cost or service. The primal-dual algorithm is the fastest method ($0.004$s) but the most expensive (objective $1{,}618{,}102$): it updates inventory dual prices online without a demand forecast, so they stay far from what the forecast-based and scenario-aware policies achieve. DTLP performs about as well as greedy on both cost and service (objective $928{,}889$; late rate $9.77\%$) and runs in $0.101$s. Both price inventory through dual variables but ignore the contextual delivery-time uncertainty that the proxy and C-SAA capture through scenario forecasts, so neither avoids the options likely to ship late, and both leave late rates well above the proxy's $4.93\%$. The proxy outperforms both on every cost and service metric at comparable per-order speed.

\begin{table}[!ht]
\centering
\caption{Aggregate out-of-sample performance for the proxy, the C-SAA reference, and the five fulfillment baselines. Totals, average late rate, and cumulative lateness are means with 95\% CIs across the 50 replications; runtimes and quantiles ($q_{50}, q_{75}, q_{90}, q_{95}$) are point estimates.}
\label{tab:sim_summary_test}
\begin{adjustbox}{width=\linewidth,center}
\begin{tabular}{lrrrrrrr}
\toprule
\textbf{Metric} & \textbf{Proxy} & \textbf{C-SAA} & \textbf{DTLP} & \textbf{Greedy} & \textbf{Empirical-SAA} & \textbf{PTO} & \textbf{Primal-Dual} \\
\midrule
Total Objective Value ($\downarrow$) & 829,805.22 $\pm$ 3,591.60 & 857,892.49 $\pm$ 3,397.22 & 928,889.34 $\pm$ 3,587.17 & 928,970.28 $\pm$ 4,066.40 & 1,015,540.50 $\pm$ 3,818.07 & 1,021,986.33 $\pm$ 3,751.13 & 1,618,102.13 $\pm$ 3,204.83 \\
Avg Late Delivery Rate (\%) ($\downarrow$) & 4.93 $\pm$ 0.07 & 5.71 $\pm$ 0.08 & 9.77 $\pm$ 0.08 & 9.80 $\pm$ 0.08 & 11.12 $\pm$ 0.08 & 12.46 $\pm$ 0.09 & 8.93 $\pm$ 0.07 \\
Avg Cumulative Lateness ($\downarrow$) & 0.122 $\pm$ 0.002 & 0.132 $\pm$ 0.002 & 0.191 $\pm$ 0.002 & 0.191 $\pm$ 0.003 & 0.226 $\pm$ 0.002 & 0.229 $\pm$ 0.002 & 0.172 $\pm$ 0.002 \\
Avg Run Time (s) ($\downarrow$) & 0.006 & 17.084 & 0.101 & 0.104 & 16.546 & 0.223 & 0.004 \\
Objective q50 ($\downarrow$) & 828,623.86 & 856,328.63 & 927,611.38 & 925,307.78 & 1,015,155.68 & 1,022,939.95 & 1,616,524.13 \\
Objective q75 ($\downarrow$) & 836,647.11 & 863,805.48 & 934,345.23 & 936,506.18 & 1,021,156.23 & 1,028,017.25 & 1,625,064.43 \\
Objective q90 ($\downarrow$) & 841,212.72 & 868,645.43 & 943,459.70 & 951,517.40 & 1,034,531.14 & 1,032,943.11 & 1,631,733.95 \\
Objective q95 ($\downarrow$) & 852,203.32 & 883,859.23 & 952,590.11 & 959,385.00 & 1,039,722.57 & 1,046,086.78 & 1,636,044.90 \\
Late Delivery q50 (\%) ($\downarrow$) & 4.87 & 5.70 & 9.74 & 9.76 & 11.07 & 12.45 & 8.90 \\
Late Delivery q75 (\%) ($\downarrow$) & 5.03 & 5.86 & 9.92 & 9.92 & 11.26 & 12.69 & 9.07 \\
Late Delivery q90 (\%) ($\downarrow$) & 5.22 & 6.02 & 10.10 & 10.19 & 11.42 & 12.78 & 9.25 \\
Late Delivery q95 (\%) ($\downarrow$) & 5.34 & 6.07 & 10.28 & 10.41 & 11.61 & 12.89 & 9.39 \\
Cumulative Lateness q50 ($\downarrow$) & 0.121 & 0.131 & 0.190 & 0.189 & 0.225 & 0.230 & 0.171 \\
Cumulative Lateness q75 ($\downarrow$) & 0.126 & 0.136 & 0.194 & 0.196 & 0.229 & 0.233 & 0.177 \\
Cumulative Lateness q90 ($\downarrow$) & 0.129 & 0.139 & 0.200 & 0.206 & 0.238 & 0.236 & 0.181 \\
Cumulative Lateness q95 ($\downarrow$) & 0.137 & 0.149 & 0.206 & 0.211 & 0.241 & 0.245 & 0.184 \\
\bottomrule
\end{tabular}
\end{adjustbox}
\end{table}

\begin{figure}[!ht]
    \centering
    \includegraphics[width=0.9\linewidth]{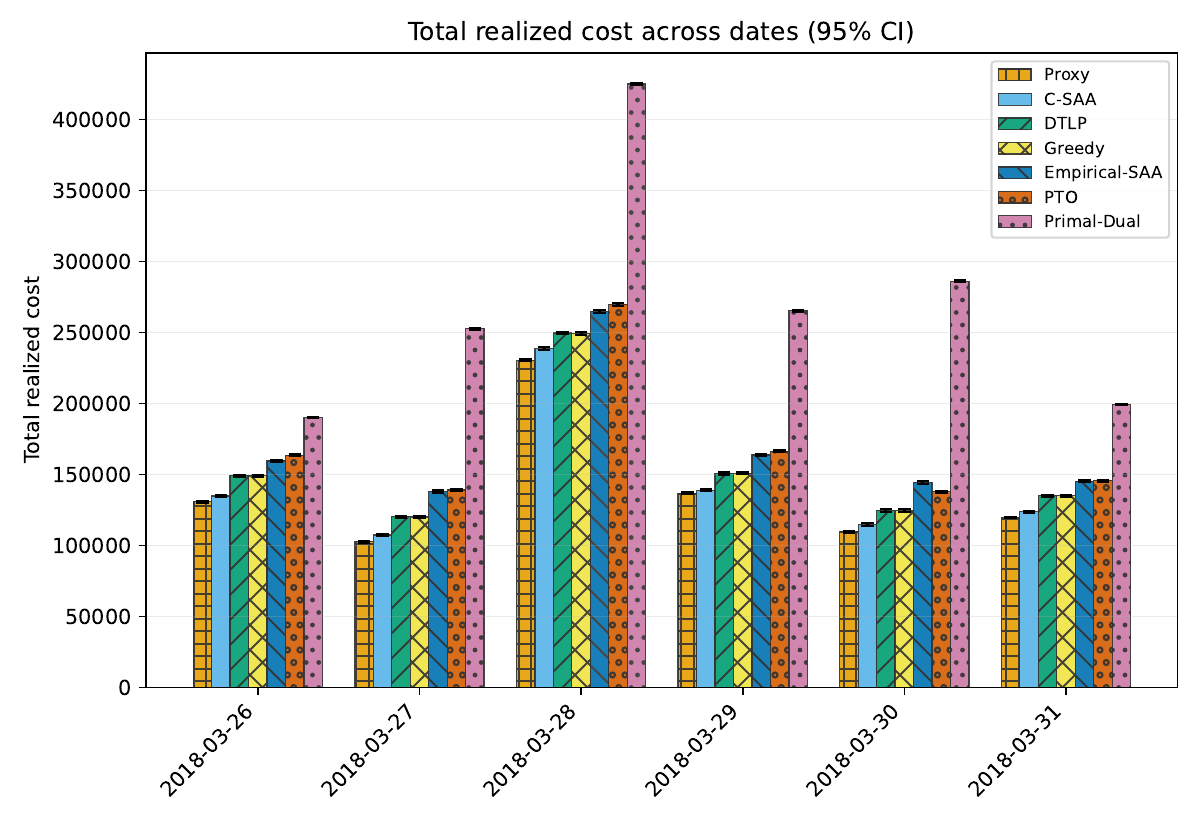}
    \caption{Total realized fulfillment cost by evaluation date, summed over orders served each day and averaged across the 50 replications.}
    \label{fig:sim_test_realized_cost_by_date_bar}
\end{figure}

Figure~\ref{fig:sim_test_realized_cost_by_date_bar} breaks the realized cost down by date, summing over the orders served each day and averaging across the 50 replications. The proxy stays at or below every baseline on each evaluation date and keeps its margin over the baseline policies across the week.

\subsection{Computational Scalability}

This Section highlights the computational performance of the proposed proxy against the baseline C-SAA approach by examining its scalability in terms of number of scenarios and size of the instances considered.

\subsubsection{Scenario Size Scalability}
Figure~\ref{fig:scenario_sensitivity_obj_runtime_combined_vs_n1} reports solution quality and latency as the scenario sample size $N_1$ varies. For the C-SAA solver, total objective value decreases monotonically with $N_1$, but average runtime grows roughly linearly to about 30 seconds at $N_1=90$. The proxy, by contrast, attains its minimum objective at the training value $N_1=50$ and degrades little for larger $N_1$, with per-decision latency remaining near 0.006 seconds across the tested $N_1$ range.

\begin{figure}[!ht]
    \centering
    \includegraphics[width=0.78\linewidth]{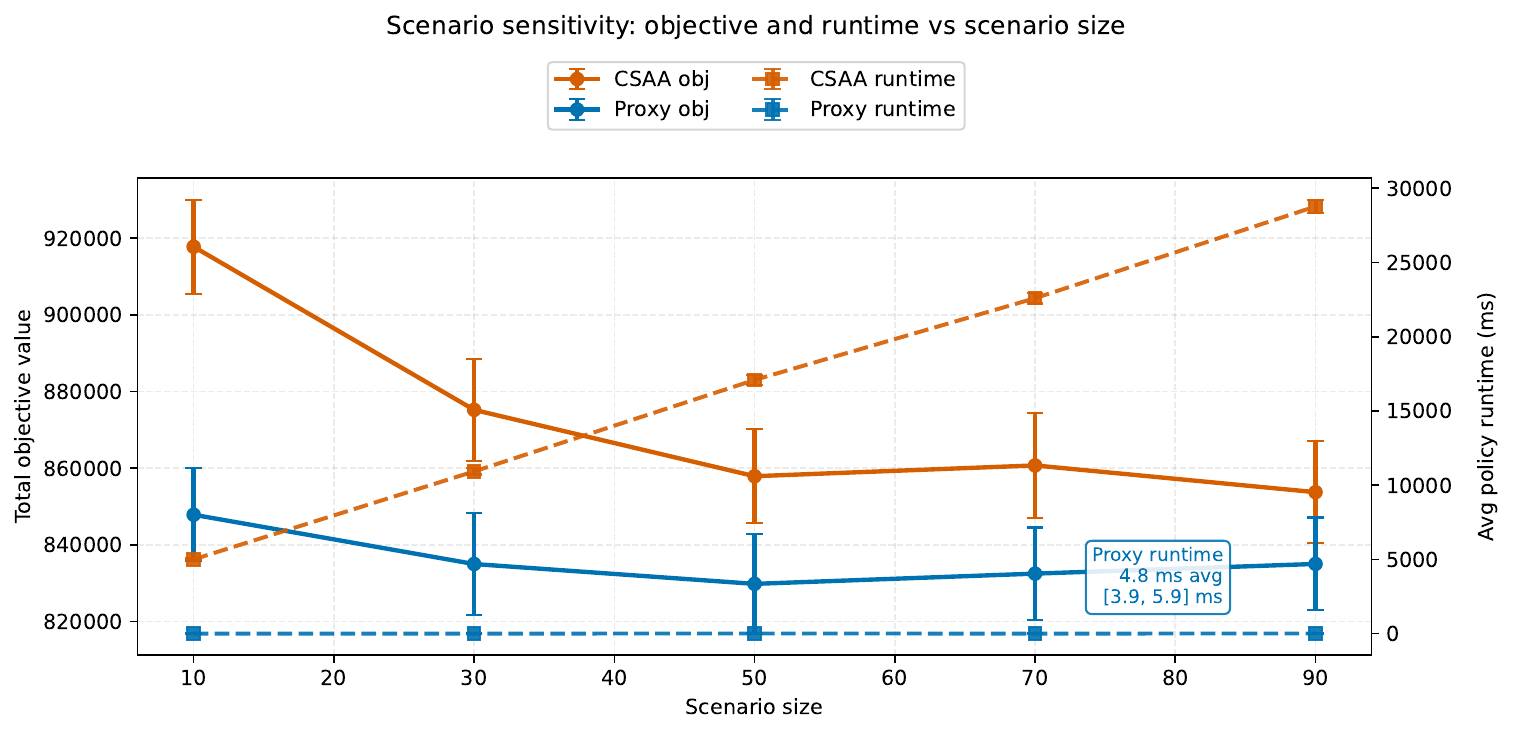}
    \caption{Scenario Size Sensitivity.}
\label{fig:scenario_sensitivity_obj_runtime_combined_vs_n1}
\end{figure}

\subsubsection{Instance Size Scalability}
Table~\ref{tab:runtime_gap_size_test} stratifies the test evaluation set by combinatorial complexity, defined as the product of the order quantity and the total number of eligible fulfillment options. As instance size grows, the C-SAA solver's runtime increases from roughly 15 to 27 seconds, while the proxy's per-decision time stays near 0.006 seconds across the tested tiers. On the largest instances the C-SAA point estimate is marginally lower than the proxy's, but the 95\% confidence intervals overlap, so the two are statistically indistinguishable on that tier. The proxy maintains a lower late delivery rate across all tiers.

\begin{table}[!ht]
\centering
\caption{Out-of-sample performance stratified by instance size, defined as the product of the order quantity and total eligible fulfillment options. Size tiers use $k$-means on a log scale with boundaries at $1{,}026$ and $1{,}927$ options.}
\label{tab:runtime_gap_size_test}
\begin{adjustbox}{width=\linewidth,center}
\begin{tabular}{llrlrrrr}
\toprule
\textbf{Size} & \textbf{Method} & \textbf{Orders} & \textbf{Size Range} & \textbf{Avg Runtime (s)} & \textbf{Avg Obj / Order ($\pm$ CI95)} & \textbf{Late \% ($\pm$ CI95)} & \textbf{Cum. Late ($\pm$ CI95)} \\
\midrule
\multirow{2}{*}{Small} & C-SAA & \multirow{2}{*}{31,435} & \multirow{2}{*}{599--1,026} & 15.479 & 14.86 $\pm$ 0.07 & 5.444 $\pm$ 0.074 & 0.097 $\pm$ 0.002 \\
 & Proxy &  &  & 0.006 & 14.16 $\pm$ 0.07 & 4.682 $\pm$ 0.071 & 0.087 $\pm$ 0.002 \\
\midrule
\multirow{2}{*}{Medium} & C-SAA & \multirow{2}{*}{5,853} & \multirow{2}{*}{1,026--1,927} & 23.880 & 44.53 $\pm$ 0.21 & 6.618 $\pm$ 0.124 & 0.221 $\pm$ 0.005 \\
 & Proxy &  &  & 0.006 & 43.30 $\pm$ 0.17 & 5.738 $\pm$ 0.100 & 0.196 $\pm$ 0.004 \\
\midrule
\multirow{2}{*}{Large} & C-SAA & \multirow{2}{*}{1,128} & \multirow{2}{*}{1,927--139,600} & 26.633 & 115.48 $\pm$ 1.04 & 8.249 $\pm$ 0.220 & 0.665 $\pm$ 0.026 \\
 & Proxy &  &  & 0.006 & 116.38 $\pm$ 1.15 & 7.513 $\pm$ 0.184 & 0.721 $\pm$ 0.029 \\
\bottomrule
\end{tabular}
\end{adjustbox}
\end{table}

\subsection{Proxy Ablations} \label{sec:proxy-ablations}
An ablation study quantifies the contribution of individual framework components across the neural architecture, loss formulation, and inference strategy. Aggregate impacts on objective cost and service reliability are summarized in Tables~\ref{tab:proxy_architecture_ablation_summary}, \ref{tab:proxy_loss_ablation_summary}, and \ref{tab:proxy_inference_strategy_ablation_summary}. These ablations also stand in for natural learning-based baselines: \texttt{no\_cost\_loss} (imitation-heavy, drops $\mathcal{L}_{\mathrm{cost}}$), \texttt{no\_selection\_loss} (cost-only, drops solver-label supervision), and \texttt{no\_scenario\_module} (no-scenario proxy, same decoder). All are inferior to the full model on objective cost.

\paragraph{Architecture and Features:}
Table~\ref{tab:proxy_architecture_ablation_summary} confirms the value of the hierarchical, multimodal design. Removing the DC module (\texttt{no\_dc\_module}) strips location-specific features and inventory state, costing 11.7\% in objective and pushing the late rate to 5.45\%. Flattening the heads (\texttt{single\_tower}) costs 11.2\% in objective with a comparable late rate, indicating that joint heads sacrifice cost economies for speed. Removing the scenario module (\texttt{no\_scenario\_module}) forces the proxy onto deterministic features and produces a consistent cost increase.

\begin{table}[!ht]
\centering
\caption{Impact of architectural components on out-of-sample performance (means $\pm$ 95\% CIs).}
\label{tab:proxy_architecture_ablation_summary}
\begin{adjustbox}{width=\linewidth,center}
\begin{tabular}{lrrrr}
\toprule
\textbf{Metric} & \textbf{full\_model} & \textbf{no\_scenario\_module} & \textbf{single\_tower} & \textbf{no\_dc\_module} \\
\midrule
\textbf{Objective Total} & 829,805.22 $\pm$ 3,591.60 & 831,981.53 $\pm$ 3,307.30 & 922,693.47 $\pm$ 3,270.59 & 926,661.69 $\pm$ 3,579.82 \\
\textbf{Late Delivery \%} & 4.93 $\pm$ 0.07 & 4.86 $\pm$ 0.08 & 4.55 $\pm$ 0.07 & 5.45 $\pm$ 0.07 \\
\textbf{Cumulative Lateness} & 0.122 $\pm$ 0.002 & 0.125 $\pm$ 0.002 & 0.115 $\pm$ 0.002 & 0.136 $\pm$ 0.002 \\
\bottomrule
\end{tabular}
\end{adjustbox}
\end{table}

\paragraph{Loss Formulation:}
Table~\ref{tab:proxy_loss_ablation_summary} isolates the contribution of each composite-loss term. Removing the selection loss (\texttt{no\_selection\_loss}) costs 3.7\% in objective, showing that direct cost-based training alone, without C-SAA label supervision, is insufficient. Removing the cost-alignment term $\mathcal{L}_{\mathrm{cost}}$ costs 3.2\% in objective and pushes the late rate to 5.63\%, isolating its role in smoothing expected scenario costs and consistent with the proxy's improvement over the finite-sample C-SAA solver. Removing the constraint loss $\mathcal{L}_{\mathrm{const}}$ costs 1.8\% in objective with comparable service-level degradation, confirming its role in shaping operationally feasible primary-DC commitments before decoding.

\begin{table}[!ht]
\centering
\caption{Impact of loss function components on out-of-sample performance (means $\pm$ 95\% CIs).}
\label{tab:proxy_loss_ablation_summary}
\begin{adjustbox}{width=\linewidth,center}
\begin{tabular}{lrrrr}
\toprule
\textbf{Metric} & \textbf{full\_model} & \textbf{no\_constraint\_loss} & \textbf{no\_cost\_loss} & \textbf{no\_selection\_loss} \\
\midrule
\textbf{Objective Total} & 829,805.22 $\pm$ 3,591.60 & 845,144.16 $\pm$ 3,743.72 & 856,460.94 $\pm$ 3,514.36 & 860,285.99 $\pm$ 3,490.07 \\
\textbf{Late Delivery \%} & 4.93 $\pm$ 0.07 & 5.62 $\pm$ 0.08 & 5.63 $\pm$ 0.07 & 5.24 $\pm$ 0.07 \\
\textbf{Cumulative Lateness} & 0.122 $\pm$ 0.002 & 0.134 $\pm$ 0.002 & 0.139 $\pm$ 0.002 & 0.128 $\pm$ 0.002 \\
\bottomrule
\end{tabular}
\end{adjustbox}
\end{table}

\paragraph{Inference Strategy:}
Table~\ref{tab:proxy_inference_strategy_ablation_summary} compares the proposed \texttt{Inv-Weighted} decoder against four alternative decoders specified in Appendix~\ref{app:decoder_variants}. \texttt{Inv-Weighted} and \texttt{Greedy-Prob} are within 0.1\% of each other in objective, but \texttt{Greedy-Prob} lacks the explicit bias toward fully feasible first choices that \texttt{Inv-Weighted}'s inventory-adequacy ratio $r_j^t$ provides. Hard-feasibility filters (\texttt{Top-K-Feasible-DC}, \texttt{Feasible-DC-First}) lose ground because they discard high-probability DCs that fall marginally short of full coverage. The joint-space search (\texttt{Top-K-Feasible-Joint}) breaks the factorized hierarchy and produces a 30.1\% cost explosion with a late rate above 11\%.

\begin{table}[!ht]
\centering
\caption{Impact of inference strategies on out-of-sample performance (means $\pm$ 95\% CIs).}
\label{tab:proxy_inference_strategy_ablation_summary}
\begin{adjustbox}{width=\linewidth,center}
\begin{tabular}{lrrrrr}
\toprule
\textbf{Metric} & \textbf{Inv-Weighted} & \textbf{Top-K-Feasible-DC} & \textbf{Greedy-Prob} & \textbf{Feasible-DC-First} & \textbf{Top-K-Feasible-Joint} \\
\midrule
\textbf{Objective Total} & 829,805.22 $\pm$ 3,591.60 & 830,412.57 $\pm$ 3,490.95 & 830,527.60 $\pm$ 3,430.36 & 837,853.57 $\pm$ 3,244.52 & 1,079,642.16 $\pm$ 3,500.57 \\
\textbf{Late Delivery \%} & 4.93 $\pm$ 0.07 & 4.93 $\pm$ 0.07 & 4.93 $\pm$ 0.07 & 4.93 $\pm$ 0.08 & 11.75 $\pm$ 0.07 \\
\textbf{Cumulative Lateness} & 0.122 $\pm$ 0.002 & 0.122 $\pm$ 0.002 & 0.122 $\pm$ 0.002 & 0.127 $\pm$ 0.002 & 0.247 $\pm$ 0.002 \\
\bottomrule
\end{tabular}
\end{adjustbox}
\end{table}

\subsection{Managerial Insights}

The computational study suggests four practical implications for large omnichannel fulfillment networks. First, the main value of the proposed approach is not just faster execution. By shifting solver compute offline, stochastic fulfillment can be deployed in real time: offline solver-based solutions supervise a learned policy that then runs online in milliseconds. Second, scenario-aware methods materially outperformed point-estimate baselines on realized cost, so scenario generation is a primary modeling lever rather than a peripheral forecasting step. Third, the inference-strategy ablation shows that how predictive outputs are translated into feasible actions has a measurable effect on realized cost and service, so inventory-aware feasible inference belongs in the policy rather than in post-processing. Finally, the results support a hybrid operating model. The learned proxy serves high-throughput real-time traffic where latency precludes online stochastic optimization, and online C-SAA remains useful for offline policy generation and selected high-stakes orders.

\section{Conclusion} \label{sec:conclusion}

This paper formulated online omnichannel order fulfillment as a two-stage contextual stochastic program and addressed the latency of its mixed-integer C-SAA reformulation through a scenario-embedded hierarchical optimization proxy paired with an inventory-weighted feasibility decoder. The network maps context and scenarios to distribution-center and carrier-service logits in a single forward pass, and the decoder turns them into a feasible per-order plan. The composite training loss augments solver-label imitation with a constraint-violation penalty and a self-supervised cost-alignment term, so the trained proxy can match or improve upon its finite-sample C-SAA reference under the same scenario forecaster. Theoretical analysis establishes that the decoder returns a feasible decision in $O(|\mathcal{J}|\log|\mathcal{J}|+|\mathcal{J}||\mathcal{K}|)$ time per product line, that its expected cost upper-bounds the two-stage stochastic optimum, and that its out-of-sample value admits unbiased, consistent sample-mean and sample-variance estimators.

The study on a simulator built from JD.com transactional records evaluated the proxy against online C-SAA and heuristic baselines under a common fitted forecaster. The proxy reduced per-order decision latency by roughly $2{,}800\times$ relative to online finite-sample C-SAA and ran in constant time as scenario density grew, whereas C-SAA latency scaled linearly. On realized fulfillment cost the proxy improved over the same reference by $3.3\%$, with the strongest gains on small-to-medium instances and statistical parity on the largest tier.

Several limitations point to natural extensions. The per-product-line decomposition keeps inference fast but captures the multi-line consolidation discount only indirectly through the jointly optimized C-SAA labels. An end-to-end architecture would broaden applicability to multi-item baskets. The second-stage recourse aggregates intra-day demand into a single approximation and would benefit from finer multi-period coupling. Methodologically, extending the uncertainty quantification of the demand and delivery-time forecasters, for instance through conformal prediction or distributionally robust forecasting, would supply coverage guarantees on the sampled scenarios under distribution shift.

\section*{Acknowledgments}
This research was supported by the NSF AI Institute for Advances in Optimization (Award 2112533).

\appendix
\numberwithin{table}{section}
\numberwithin{figure}{section}
\numberwithin{equation}{section}

\section{Nomenclature} \label{app:nomen}

Table~\ref{tab:nomenclature} summarizes the notation, sets, and mathematical operators employed throughout this paper.

\begin{small}
\begin{longtable}{p{4.0cm}p{12cm}}
\caption{Summary of Notation and Symbols} \label{tab:nomenclature} \\
\hline
\textbf{Symbol} & \textbf{Description} \\
\hline
\endfirsthead
\hline
\textbf{Symbol} & \textbf{Description} \\
\hline
\endhead
\hline
\multicolumn{2}{r}{\emph{Continued on next page}} \\
\endfoot
\hline
\endfoot
\hline
\endlastfoot
\multicolumn{2}{l}{\emph{Sets, indices, horizon}}\\
$\mathcal{T}=\{1,\ldots,T\}$ & Planning horizon consisting of $T$ epochs, where $t$ indexes individual epochs. \\
$\mathcal{I}, i, I{=}\lvert\mathcal{I}\rvert$ & Set of products with index $i$ and cardinality $I$. \\
$\mathcal{J}, j, J{=}\lvert\mathcal{J}\rvert$ & Fulfillment locations (distribution centers) with index $j$ and cardinality $J$. \\
$\mathcal{K}, k, K{=}\lvert\mathcal{K}\rvert$ & Set of carrier services with index $k$ and cardinality $K$. \\
$\mathcal{X}$ & Contextual covariate space. \\
$\tilde{\Omega}^t$, $E$, $N_1$, $N_2$, $S$ & $E$ sampled scenarios at epoch $t$. C-SAA generates $S$ candidate first-stage decisions on independent samples of size $N_1$ each, and evaluates them on a common sample of size $N_2 \gg N_1$. \\
$\tilde{\Omega}^t_{i,\mathrm{dem}}, \tilde{\Omega}^t_{j,k,\mathrm{time}}$ & Per-product demand scenarios and per-(DC, carrier) delivery-time scenarios at epoch $t$, respectively. \\
$\mathcal{N}, n$ & Training set of offline historical order instances indexed by $n$. \\
\hline
\multicolumn{2}{l}{\emph{Order, state, inventory}}\\
$q^t, q^n$ & Units requested for the product line at epoch $t$ (online) or instance $n$ (offline). \\
$\mathrm{Inv}_{i,j}^t, \mathrm{Inv}_{j}^n$ & On-hand inventory level at location $j$ for product $i$ during epoch $t$ (online) or for instance $n$ (offline). \\
$\mathbf{Inv}^n$ & Vector form of on-hand inventory across $\mathcal{J}$ for instance $n$. \\
$\mathbf{X}^t$ & Contextual features observed at epoch $t$. \\
\hline
\multicolumn{2}{l}{\emph{Decisions}}\\
$\mathbf{z}^t$ & First-stage fulfillment decision for the order at epoch $t$, specifying the units of product $i$ from location $j$ via carrier $k$. \\
$\mathbf{u}^t$ & Stockout slack decision representing unfulfilled units at epoch $t$. \\
$\mathcal{Z}^t$ & Feasible decision space for the fulfillment problem at epoch $t$. \\
$\mathbf{v}^t, \mathbf{w}^t$ & Aggregated second-stage fulfillment and stockout decisions for future periods. \\
$\mathbf{v}^t_{\omega}, \mathbf{w}^t_{\omega}$ & Second-stage allocation and stockout decisions under scenario $\omega$. \\
$\mathbf{z}^{\text{proxy}, t}$ & First-stage fulfillment decision generated by the neural optimization proxy. \\
\hline
\multicolumn{2}{l}{\emph{Uncertain parameters}}\\
$\tilde{\boldsymbol{\Delta}}^t$ & This random vector represents delivery-time deviations when an order is served from location $j$ via carrier $k$. \\
$\Delta^t_{j,k}$ & Realized delivery-time deviation, where positive values indicate lateness and negative values indicate early delivery. \\
$\tilde{\mathbf{D}}^t$ & This random vector represents the cumulative future demand from epoch $t+1$ to $T$. \\
$\boldsymbol{\Delta}^t_{\omega}, D^t_{i,\omega}$ & Realized delivery-time deviations and future demands under scenario $\omega$. \\
\hline
\multicolumn{2}{l}{\emph{Costs and objectives}}\\
$c_{i,j,k}^{\mathrm{ship}}$ & Deterministic base shipping cost for product $i$ from location $j$ via carrier $k$. \\
$\beta$ & Consolidation discount applied when multiple units ship from the same location. \\
$\gamma^+, \gamma^-$ & Asymmetric unit penalties applied for late and early deliveries, respectively. \\
$\rho$ & Unit penalty associated with stockouts or lost sales. \\
$C(\mathbf{z}^t, \mathbf{u}^t, \tilde{\boldsymbol{\Delta}}^t)$ & Per-epoch cost, including shipping costs, discounts, and stochastic penalties. \\
$Q(\mathbf{z}^t, \tilde{\mathbf{D}}^t)$ & Expected future fulfillment cost in the second stage given the first-stage decision. \\
\hline
\multicolumn{2}{l}{\emph{Distributions and operators}}\\
$\mathbb{P}(\cdot\,|\,\mathbf{X}^t)$ & Conditional probability law of the uncertainties given the contextual features. \\
$[x]^+$ & Standard rectified linear unit (ReLU) function defined as $\max(0, x)$. \\
$\operatorname{vec}(\cdot)$ & Flattens a matrix or tensor into a column vector under a fixed (column-major) ordering of its entries. \\
$\mathrm{stats}(\cdot)$ & Eight-dimensional cross-carrier summary aggregator: mean, minimum, standard deviation, $90$th percentile, best-vs-second gap of base shipping cost over carriers, plus three delivery-penalty statistics over carriers and scenarios. \\
$\mathrm{softmax}(\mathbf{x})_i$ & Softmax operator $\mathrm{softmax}(\mathbf{x})_i = \exp(x_i) / \sum_j \exp(x_j)$, mapping a logit vector to the probability simplex. \\
\hline
\multicolumn{2}{l}{\textbf{Optimization-proxy notation}}\\
\multicolumn{2}{l}{\quad\emph{Inputs}}\\
$\mathbf{X}_{\mathrm{ord}}^n$ & Order-context input at instance $n$: numeric order, calendar, and customer fields together with the SKU and brand categorical identifiers. \\
$\mathbf{X}_{\mathrm{dc},j}^n$ & Per-DC features at instance $n$: on-hand inventory $\mathrm{Inv}^n_j$, days-of-supply, region match, consolidation potential, distance, and rolling 2-hour shipping queue statistics, together with the DC categorical identifier. \\
$\mathbf{X}_{\mathrm{cost},j,k}^n$ & Per-option (DC--carrier) features at instance $n$: the deterministic base shipping cost $c^{\mathrm{ship}}_{j,k}$, scenario-derived delivery-penalty summary statistics (mean, standard deviation, and 90th percentile), and the carrier categorical identifier. \\
$\mathbf{\Omega}^n_{\mathrm{dem}} = \{D^n_\omega\}_{\omega=1}^E$ & Sampled cumulative-demand scenarios for the product line of instance $n$. \\
$\mathbf{\Omega}^n_{\mathrm{time}} = \{\boldsymbol{\Delta}_\omega^n\}_{\omega=1}^E$ & Sampled delivery-time deviation scenarios; each $\boldsymbol{\Delta}_\omega^n \in \mathbb{R}^{|\mathcal{J}| \times |\mathcal{K}|}$. \\
$\mathbf{\Omega}^n_{\mathrm{cost}} = \{\mathbf{c}_\omega^n\}_{\omega=1}^E$ & Realized scenario costs derived from $\mathbf{\Omega}^n_{\mathrm{time}}$ and per-option features, used by the cost-alignment loss. \\
$j^{n,*},\, k^{n,*},\, \mathbf{z}^{n,*}$ & Primary location, carrier-service, and full first-stage labels extracted from the C-SAA solution for instance $n$, used as training targets. \\
\hline
\multicolumn{2}{l}{\quad\emph{Outputs}}\\
$\boldsymbol{\ell}^{\mathrm{dc}} \in \mathbb{R}^{|\mathcal{J}|}$ & This vector contains the DC logits returned by the proxy. \\
$\boldsymbol{\ell}^{\mathrm{carrier}} \in \mathbb{R}^{|\mathcal{J}| \times |\mathcal{K}|}$ & This matrix contains the carrier-service logits returned by the proxy. \\
\hline
\multicolumn{2}{l}{\quad\emph{Learnable parameters}}\\
$f_\theta,\, \theta$ & Proxy network and the collection of all learnable weights, including categorical embedding tables for SKU, brand, DC, and carrier identifiers. \\
$\phi_{\mathrm{ord}},\, \phi_{\mathrm{dc}},\, \phi_{\mathrm{cost}}$ & Per-modality embedding MLPs for the order context, per-DC features, and the per-DC cost summary, respectively. \\
$\phi_{\mathrm{scen}}^{\mathrm{dem}},\, \phi_{\mathrm{scen}}^{\mathrm{time}}$ & Per-scenario encoders for demand and delivery-time samples. \\
$g_{\mathrm{dc}},\, g_{\mathrm{carrier}}$ & DC and carrier head MLPs. \\
\hline
\multicolumn{2}{l}{\quad\emph{Auxiliary values}}\\
$\mathbf{e}_{\mathrm{global}},\, \mathbf{e}_{\mathrm{scen}},\, \mathbf{e}_{\mathrm{dc},j},\, \mathbf{e}_{\mathrm{cost},j}$ & Embeddings produced by the four input modules. \\
$\mathbf{m}^{\mathrm{dem}},\, \mathbf{m}^{\mathrm{time}}$ & Per-branch scenario-pooled embeddings before fusion. \\
$\mathbf{h}_j$ & Joint per-DC embedding constructed by concatenation. \\
$\mathbf{p}^{\mathrm{dc}},\, \mathbf{p}^{\mathrm{carrier}}$ & These vectors hold the softmax probabilities derived from $\boldsymbol{\ell}^{\mathrm{dc}}$ and $\boldsymbol{\ell}^{\mathrm{carrier}}$. \\
$\mathbf{a}^n$ & One-hot sample of $\boldsymbol{\ell}^{\mathrm{dc}}$: hard one-hot in the forward pass, temperature-$\tau$ softmax relaxation in the backward pass (Straight-Through Gumbel-Softmax estimator, \citealt{jang2017categorical}). \\
$r^t_j \in [0, 1]$ & Inventory adequacy ratio used by the deterministic decoder during inference. \\
\hline
\multicolumn{2}{l}{\quad\emph{Loss components and parameters}}\\
$\mathcal{L}_{\mathrm{sel}},\, \mathcal{L}_{\mathrm{const}},\, \mathcal{L}_{\mathrm{cost}}$ & Composite-loss components: selection (C-SAA-label imitation), constraint penalty (joint demand-balance and inventory-cap shortfall), and self-supervised cost alignment. \\
$\lambda_{\mathrm{sel}},\, \lambda_{\mathrm{carrier}},\, \lambda_{\mathrm{const}},\, \lambda_{\mathrm{cost}}$ & These scalars balance the terms in the composite training loss. \\
$\tau$ & Temperature of the softmax relaxation used in the backward pass of $\mathbf{a}^n$. \\
\hline
\end{longtable}
\end{small}

\section{C-SAA Evaluation Procedure}\label{app:csaa_procedure}
Algorithm~\ref{alg:saa} summarizes the resulting candidate-generation and evaluation procedure, where \textsc{SolveSAA}$(\Omega)$ denotes solving (2) with scenario set $\Omega$ and returns the first-stage solution and the resulting optimal objective function value.

\begin{algorithm}[!ht]
\caption{Two-Stage C-SAA with Candidate Solution Generation and Evaluation}
\label{alg:saa}
\begin{algorithmic}[1]
\Require Number of candidates $S$; generation scenario size $N_1$; evaluation scenario size $N_2$
\Ensure Selected solution $(\mathbf{z}^{t}_{s^\star}, \mathbf{u}^{t}_{s^\star})$ and estimated objective $\bar f_{s^\star}$

\Procedure{GenerateCandidates}{$S, N_1$}
    \For{$s = 1$ to $S$} \Comment{Generate candidate solutions}
        \State $\Omega_{s,N_1} \gets \textsc{SampleScenarios}(N_1)$
        \Comment{Sample demand and delivery-time scenarios}
        \State $(\mathbf{z}^t_s, \mathbf{u}^t_s, f_{s,N_1}) \gets \textsc{SolveSAA}(\Omega_{s,N_1})$
    \EndFor
    \State \Return $\{(\mathbf{z}^t_s, \mathbf{u}^t_s)\}_{s=1}^{S}$
\EndProcedure

\Procedure{EvaluateCandidates}{$\{(\mathbf{z}^t_s, \mathbf{u}^t_s)\}_{s=1}^{S}, N_2$}
    \State $\Omega_{N_2} \gets \textsc{SampleScenarios}(N_2)$
    \Comment{Sample a larger scenario set for evaluation}
    \For{$s = 1$ to $S$} \Comment{Evaluate each candidate}
        \For{each $\omega \in \Omega_{N_2}$}
            \State $\phi_{s,\omega} \gets
            Q(\mathbf{z}^t_s, \mathbf{D}^t_\omega)$
        \EndFor
        \State $\displaystyle
        \bar f_s \gets \frac{1}{N_2}
        \sum_{\omega \in \Omega_{N_2}}
        \Big( \phi_{s,\omega} + C(\mathbf{z}^t_s, \mathbf{u}^t_s, \boldsymbol{\Delta}^t_{\omega}) \Big)$
        \Comment{Estimated expected cost of candidate $s$}
    \EndFor
    \State \Return $\{\bar f_s\}_{s=1}^{S}$
\EndProcedure

\State $\{(\mathbf{z}^t_s, \mathbf{u}^t_s)\}_{s=1}^{S} \gets \textsc{GenerateCandidates}(S,N_1)$
\State $\{\bar f_s\}_{s=1}^{S} \gets \textsc{EvaluateCandidates}(\{(\mathbf{z}^t_s, \mathbf{u}^t_s)\}_{s=1}^{S}, N_2)$
\State $s^\star \gets \arg\min_{s \in \{1,\dots,S\}} \bar f_s$
\Comment{Select the best candidate}
\State \Return $(\mathbf{z}^t_{s^\star}, \mathbf{u}^t_{s^\star}, \bar f_{s^\star})$
\end{algorithmic}
\end{algorithm}

\section{Fulfillment Instance Setup} \label{app:instance}
This section details the construction of the daily fulfillment instances, including the generation of feasible options, state initialization, and the composite realized cost structure used in the simulation environment.

\subsection{Simulation Scope and State Initialization}
The simulation environment models independent daily instances rather than a continuous multi-day rolling horizon. At the start of each peak operational window, the system state is freshly initialized, meaning no inventory or queue backlogs carry over between simulation dates.
\begin{itemize}
    \item \textbf{Inventory Initialization:} Initial on-hand inventory is computed separately for each simulation date using historical demand strictly preceding that date. For a given facility and item, the target stock level follows the standard safety-stock formulation $\text{target} = \mu_{\text{daily}} + z \cdot \sigma_{\text{daily}}$, where $\mu_{\text{daily}}$ and $\sigma_{\text{daily}}$ are the mean and standard deviation of historical daily demand. For local distribution centers, the $z$-factor corresponds to a 50\% cycle service level, and items are additionally filtered through a deterministic 75\% stocking probability hash. Central warehouses aggregate demand regionally and use a higher 80\% cycle service level. The simulator does not model intra-day replenishment; inventory strictly depletes as sequential fulfillment decisions allocate units.
    \item \textbf{Queue State Warm-starting:} Before processing the first order of the day, the simulator warm-starts the operational queue (e.g., waiting orders, waiting units, and pending shipments) using a historical snapshot from the preceding hours. This ensures the simulator begins from a realistic, non-empty operational state. Subsequent processing delays are estimated dynamically based on base unit processing times and queue backlogs, capped at a maximum of 720 minutes.
\end{itemize}

\subsection{Feasible Options and Base Costs}
For each arriving order, feasible fulfillment options are restricted to distribution centers and carrier services capable of delivering to the customer's geographic state. The shipping distance from the origin DC to the synthetic customer location is computed using a vectorized Haversine formula based on hub-based pseudo-coordinates. The deterministic base shipping cost for an eligible (DC, carrier) pair is $\text{Base Cost} = (\text{Distance in km}) \cdot \alpha_{\text{carrier}} + \text{Fixed DC Cost}$, where $\alpha_{\text{carrier}}$ is the externally calibrated carrier distance coefficient and the fixed facility operational cost is 2.0 for local DCs and 4.0 for central warehouses. The remaining cost-function parameters of the main-paper definition take the values $\beta = 0.5$ (consolidation discount), $\gamma^+ = 40.0$ and $\gamma^- = 0.2$ (per-unit-per-day late and early penalties), and $\rho = 200.0$ (per-unit lost-sales penalty).

\subsection{Delivery Time Simulator Construction}

Because realized delivery-time deviations are only observed for the location--carrier options selected in the historical system, counterfactual delivery outcomes are unavailable for alternative fulfillment decisions. Following \citet{ye2025contextual}, a separate counterfactual simulator is therefore constructed to support online policy evaluation. The simulator uses the same multi-quantile MLP (MQMLP) class and feature pipeline as the delivery-time forecaster (defined in Appendix~\ref{app:nn}), but is invoked at a different point in the loop and on disjoint data: the simulator is fit on the historical realized-delivery records and is queried only \emph{after} a policy decision is made (to score that decision), whereas the forecaster's role at decision time is to produce the scenario set $\mathbf{\Omega}^t_{\mathrm{time}}$ that policies see as input. The simulator therefore acts as a shared counterfactual evaluator: every policy compared in Section~5.3 consumes the same scenario inputs from the forecaster and is scored by the same simulator-sampled realizations, so the comparison among policies is consistent within the simulator, while not eliminating action-dependent bias of the simulator itself; calibration diagnostics by carrier and policy-relevant action region are an item for future work. For a realized context $\mathbf{X}^t$ and a selected location--carrier pair $(j,k)$, the simulator produces conditional quantile estimates that define an approximate distribution for the delivery-time deviation, from which a realization is sampled to evaluate the chosen action. Because the simulator is a fitted predictive model rather than the true delivery process, it provides an approximate counterfactual evaluation environment for comparing policies.

\section{Data Preprocessing and Carrier-Service Augmentation} \label{app:prep}

\subsection{Raw Data Cleaning}
The foundational dataset is JD.com e-commerce transaction records \citep{shen2024jd}. Several preprocessing filters were applied to clean the inputs for the simulation and forecasting models. Entries with missing delivery promise days and duplicate \texttt{(order\_ID, sku\_ID)} pairs were removed. Records with missing \texttt{brand\_ID}, single-item gift orders, and multi-package orders (identifiable by more than one unique \texttt{package\_ID} per order) were excluded to focus the evaluation on standard multi-item consolidation dynamics. 

Timestamps were parsed to compute the end-to-end delivery time (from \texttt{order\_time} to \texttt{arr\_time}) and the DC-to-DC travel time (from \texttt{ship\_out\_time} to \texttt{arr\_station\_time}). Records exhibiting negative durations due to logging errors were discarded. Delivery hours were converted to discrete days via $\lceil \text{hours}/24 \rceil$ (counting same-day delivery as 1 day). The delivery deviation was explicitly defined as \texttt{delivery\_time\_days} minus \texttt{promise\_delivery\_days}. Finally, to remove extreme outliers, records with a \texttt{delivery\_time\_days} exceeding 5 days or a deviation falling outside the $[-5, 5]$ day window were filtered out.

\subsection{Carrier-Service Augmentation}
The baseline JD.com dataset is predominantly serviced by a single internal carrier and lacks precise geographic coordinates for its fulfillment facilities. To train and evaluate the hierarchical proxy and carrier-specific simulators, a calibrated transformation layer is applied to the cleaned data to generate a synthetic, carrier-aware spatial dataset. The two augmentation parameter families described below (the per-distance-bin carrier-eligibility shares and the per-carrier delivery-time multipliers) are calibrated from a separate proprietary multi-carrier fulfillment dataset rather than chosen ad hoc, so that the augmented benchmark preserves the joint distribution of distance, carrier, and delivery-time variability observed in real operations even though the underlying transaction stream remains the public JD.com release. Aggregate calibration footprint (no proprietary records released): 17 carrier services, 5 origin-destination distance bins, an eligibility share per (distance bin, carrier) pair, and a median delivery-time ratio per (carrier, distance bin) pair (so $17 \times 5 = 85$ shares and $85$ ratios).

The augmentation proceeds sequentially. First, distribution centers are assigned synthetic coordinates based on regional metadata. Synthetic customer locations are similarly generated by mapping the original destination facility's three-digit (ZIP3) postal codes to consistent five-digit (ZIP5) codes, enabling precise geocoding into latitude and longitude coordinates. Second, the shipping distance from the origin distribution center to the synthetic customer location is computed utilizing a vectorized Haversine formula. Third, each order is categorized into a discrete distance bin. Fourth, a carrier-service identifier is sampled from the proprietary-calibrated eligibility pool for the corresponding distance bin. Finally, a carrier-level delivery-time perturbation is applied: the original observed delivery time is multiplied by the proprietary-calibrated median delivery-time ratio for the sampled (carrier, distance) pair. 

This process yields a deterministic hybrid dataset (controlled via a fixed random seed) that preserves the anchoring of the original transactional observations while introducing spatial coordinates and re-scaling delivery times to reflect the realistic stochastic variations induced by different carrier selections.

\section{Forecasting Model Feature Details} \label{app:features}
\subsection{Demand Forecasting Features}
The demand pipeline constructs rolling-window tensors using a historical lookback of 36 hours and a forecast horizon of 24 hours. The historical covariate tensor consists of 23 features per time step: 4 cyclical calendar variables (sine and cosine transformations of hour and day of the week) and 19 aggregated order covariates (including SKU counts, bundle and coupon discount indicators, total quantity, and customer demographic encodings). The historical target tensor contains the realized demand quantity per time step. The future covariate tensor is strictly limited to the 4 cyclical calendar variables over the prediction horizon. Missing numerical values are imputed with a constant ($-1$), and discrete covariates are factorized into integer indices. Baseline tree-based models flatten these temporal tensors into summary statistics (mean and standard deviation) or use the fully flattened temporal sequences.

\subsection{Delivery-Time Forecasting Features}
The delivery-time models use a shared feature set constructed from preprocessed order data. The target variable is the realized delivery duration in days. Numerical features include temporal indicators, order complexity metrics, and DC-operations snapshots (specifically, the volume of shipped orders and SKUs in the preceding two hours, alongside current waiting queues). Categorical features are limited to the origin and destination facility identifiers, which are label-encoded and embedded in the deep learning pipelines.

\section{Forecasting Model Performance}
Quantile regression models were trained and evaluated for both delivery-time and demand forecasting. The candidate model families are \texttt{XGBoost}'s \texttt{XGBRegressor} (XGBoost), \texttt{scikit-learn}'s \texttt{QuantileRegressor} (QR), \texttt{quantile-forest}'s \texttt{RandomForestQuantileRegressor} (QRF), an MLP-based MQMLP for delivery time, and a Multi-Quantile Recurrent Neural Network (MQRNN) for demand. All models predict 19 quantiles (0.05 to 0.95 in increments of 0.05); predictions are rounded to integer values before scoring against the ground truth. Reported metrics are pinball loss and Continuous Ranked Probability Score (CRPS), averaged across the 17 carriers and 19 quantiles for delivery time (Table~\ref{tab:delivery_metrics}) and across 24 future horizons and 19 quantiles for demand (Table~\ref{tab:demand_metrics}). For demand, MQRNN processes full temporal sequences through its recurrent architecture, whereas XGBoost, QR, and QRF use summary statistics (mean and standard deviation) of the temporal features.

\begin{table}[!ht]
\centering
\caption{Metrics for Delivery Forecasting Models Averaged Across Quantiles of 17 Carriers.}
\label{tab:delivery_metrics}
\begin{tabular}{l r r r r}
\toprule
Model & QR & XGBoost & QRF & MQMLP (proposed) \\
\midrule
Pinball Loss & 0.234 & 0.228  & 0.210 & \textbf{0.201} \\
CRPS      & 0.449  & 0.436  & 0.407 & \textbf{0.396} \\
\bottomrule
\end{tabular}
\end{table}

\begin{table}[!ht]
\centering
\caption{Metrics for Demand Forecasting Models Averaged Across Quantiles Over 24 Horizons.}
\label{tab:demand_metrics}
\begin{tabular}{l r r r r}
\toprule
Model & QR & XGBoost  & QRF & MQRNN (proposed) \\
\midrule
Pinball Loss & 0.319  & 0.237  & 0.270 & \textbf{0.229} \\
CRPS & 0.620  & 0.456  & 0.521 & \textbf{0.443} \\
\bottomrule
\end{tabular}
\end{table}

\section{Neural Network Details}\label{app:nn}

\subsection{Delivery Time Forecasting Model (MQMLP)}
The multi-quantile MLP (MQMLP) delivery time model is a feed-forward MLP that predicts multiple delivery-time quantiles. Inputs consist of numerical features and origin/destination DC IDs; the DC IDs are mapped through learned embeddings and concatenated to the numerical features before entering the network. The output layer returns one value per requested quantile.

\paragraph{Non-crossing quantiles.}
To avoid quantile crossing, the head is built around the median ($0.5$) as an anchor. The model predicts (i) a median value and (ii) nonnegative “gaps” for higher and lower quantiles via a \texttt{softplus} transform. Higher quantiles are the median plus cumulative upward gaps; lower quantiles are the median minus cumulative downward gaps. This construction produces nondecreasing quantiles by design.

\paragraph{Bounded support.}
Predictions are mapped to the fixed interval $[t_{\min}, t_{\max}] = [1, 5]$ via $\hat{q} = t_{\min} + (t_{\max}-t_{\min}) \cdot \sigma(\text{raw})$ with a monotone sigmoid $\sigma$.

\paragraph{Training objective.}
Training uses the pinball (quantile) loss summed over all predicted quantiles. Early stopping is applied on a validation split.

\paragraph{Hyperparameters.}
Because the models are trained separately for each carrier, the selected hyperparameter configuration may vary by carrier. Table~\ref{tab:carrier_hparam_space} reports the hyperparameter search space. For each model, 500 unique hyperparameter combinations are sampled, with learning rate and weight decay drawn on a log scale. For each carrier, the final hyperparameter setting is selected as the configuration attaining the lowest validation loss.

\begin{table}[!ht]
\centering
\caption{Hyperparameter search space for the carrier-specific models.}
\label{tab:carrier_hparam_space}
\small
\begin{tabular}{ll}
\toprule
\textbf{Hyperparameter} & \textbf{Search space} \\
\midrule
Hidden width   & $\{96,\,128,\,192\}$ \\
MLP layers     & $\{2,\,3,\,4\}$ \\
Dropout        & uniform $[0.15,\,0.40]$ \\
Batch size     & $\{64,\,128,\,256\}$ \\
Learning rate  & log-uniform $[3\times10^{-4},\,3\times10^{-3}]$ \\
Weight decay   & log-uniform $[3\times10^{-4},\,3\times10^{-3}]$ \\
\bottomrule
\end{tabular}
\end{table}

\subsection{Demand Forecasting (MQRNN)}
The MQRNN demand model follows an encoder–decoder design. Two input encoders map (i) the history window (calendar features, order-statistics, and SKU/brand IDs via learned embeddings) and (ii) the prediction window (calendar features and the same IDs) into compact contexts. An LSTM then processes the history context concatenated with the observed target series. A two-level decoder produces multi-horizon quantiles: a global MLP builds horizon-agnostic and horizon-specific contexts, and a local head refines each horizon.

\paragraph{Non-crossing \& non-negative quantiles.}
The quantile head is centered at the median as an anchor. It predicts nonnegative “gaps” for higher and lower quantiles using a \texttt{softplus} transform and accumulates them outwards from the median, which ensures ordered (non-crossing) quantiles. A final \texttt{ReLU} enforces non-negative forecasts.

\paragraph{Training objective.}
Training uses the pinball (quantile) loss summed over all horizons and quantiles, with early stopping on a validation split.

\paragraph{Hyperparameters.}
Table~\ref{tab:mqrnn_hparam_space} summarizes the hyperparameter search space and the selected best configuration for the MQRNN model. A total of 50 unique hyperparameter combinations are sampled, with learning rate and weight decay drawn on a log scale, consistent with the tuning procedure used for the delivery model. Training is run for up to 200 epochs, and early stopping is applied using the patience value selected from the search space.

\begin{table}[htbp]
\centering
\caption{Hyperparameter search space and selected best configuration for the MQRNN model.}
\label{tab:mqrnn_hparam_space}
\small
\begin{tabular}{lll}
\toprule
\textbf{Hyperparameter} & \textbf{Search space} & \textbf{Selected best value} \\
\midrule
Hidden dimension          & $\{256,\,384,\,512\}$ & $512$ \\
LSTM layers               & $\{3,\,4\}$ & $3$ \\
Dropout                   & uniform $[0.65,\,0.80]$ & $0.725$ \\
Learning rate             & log-uniform $[5\times10^{-5},\,1.5\times10^{-4}]$ & $5\times10^{-5}$ \\
Weight decay              & log-uniform $[5\times10^{-5},\,2\times10^{-4}]$ & $5\times10^{-5}$ \\
Batch size                & $\{16,\,32\}$ & $16$ \\
Context dimension         & $\{12,\,16,\,20\}$ & $20$ \\
SKU embedding dimension   & $\{8,\,12\}$ & $12$ \\
Brand embedding dimension & $\{12,\,16\}$ & $12$ \\
Early stopping patience   & $\{8,\,12,\,16\}$ & $8$ \\
\bottomrule
\end{tabular}
\end{table}

\subsection{The Optimization Proxy}

This appendix provides concrete dimensions and feature inventories for the symbolic specification of $f_\theta$ given in Section~4.2.1 (Figure~3, with notation defined in Table~\ref{tab:nomenclature}). Table~\ref{tab:proxy_features} lists the constituent fields of each input embedding source: categorical fields use the lookup widths shown in the table and continuous fields are min--max scaled prior to entering the network. Table~\ref{tab:proxy_module_spec} reports the per-module input source, hidden-layer structure, and output dimension. Table~\ref{tab:proxy_hyperparameters} reports the hyperparameter search space and the selected production configuration; the hidden width $h$, dropout probability, and number of head hidden layers are tuned, while the structural choices in Table~\ref{tab:proxy_module_spec} are fixed.
\begin{table}[!ht]
\centering
\caption{Feature inventory of the optimization proxy.}
\label{tab:proxy_features}
\small
\begin{tabular}{p{0.20\linewidth} p{0.13\linewidth} p{0.55\linewidth}}
\toprule
\textbf{Source} & \textbf{Type} & \textbf{Constituent fields} \\
\midrule
$\mathbf{X}_{\mathrm{ord}}^n$ (order context) & categorical & SKU id ($d_{\mathrm{sku}}=8$), brand id ($d_{\mathrm{brand}}=6$) \\
& continuous & calendar (hour, weekday), order summary (number of SKUs, total quantity, discount and gift-item indicators, order type, average discount), customer attributes (user/city level, Plus membership, promised delivery days, and demographic fields), destination DC \\
\midrule
$\mathbf{X}_{\mathrm{dc},j}^n$ (per-DC) & categorical & DC id ($d_{\mathrm{dc}}=64$) \\
& continuous & on-hand inventory $\mathrm{Inv}^n_j$, days-of-supply, region match, consolidation potential, distance (km), and dynamic operational metrics (shipped orders and SKUs in the last 2 hours, current waiting-orders and waiting-SKUs queues) \\
\midrule
$\mathbf{X}_{\mathrm{cost},j,k}^n$ (per-option) & categorical & carrier id ($d_{\mathrm{carrier}}=16$) \\
& continuous & base shipping cost $c^{\mathrm{ship}}_{j,k}$ and the mean / standard deviation / 90th percentile of scenario-realized delivery-penalty costs across the sampled delivery-time scenarios \\
\midrule
$\mathbf{\Omega}^n_{\mathrm{dem}}$ & continuous & sampled cumulative demand $D_\omega^n$ for the product line under each scenario $\omega \in \{1,\dots,E\}$ \\
\midrule
$\mathbf{\Omega}^n_{\mathrm{time}}$ & continuous & sampled per-option delivery deviations $\boldsymbol{\Delta}_\omega^n \in \mathbb{R}^{|\mathcal{J}|\times|\mathcal{K}|}$ for each scenario $\omega$ \\
\bottomrule
\end{tabular}
\end{table}

\begin{table}[!ht]
\centering
\caption{Module specification for the hierarchical optimization proxy.}
\label{tab:proxy_module_spec}
\small
\begin{tabular}{l p{0.30\linewidth} p{0.30\linewidth} l}
\toprule
\textbf{Module} & \textbf{Input} & \textbf{Hidden structure} & \textbf{Output} \\
\midrule
$\phi_{\mathrm{ord}}$ & $\mathbf{X}_{\mathrm{ord}}^n$ (categorical embeddings $\parallel$ continuous) & 1 hidden layer of width $h$ & $\mathbb{R}^{h}$ \\
$\phi_{\mathrm{scen}}^{\mathrm{dem}}$ & $D_\omega^n \in \mathbb{R}$ (per scenario) & 2 hidden layers of width $h$, light block & $\mathbb{R}^{h}$ \\
$\phi_{\mathrm{scen}}^{\mathrm{time}}$ & $\mathrm{vec}\,\boldsymbol{\Delta}_\omega^n \in \mathbb{R}^{|\mathcal{J}||\mathcal{K}|}$ (per scenario) & 2 hidden layers of width $h$, light block & $\mathbb{R}^{h}$ \\
$\phi_{\mathrm{dc}}$ & $\mathbf{X}_{\mathrm{dc},j}^n$ (DC embedding $\parallel$ continuous) & 1 hidden layer of width $h$ & $\mathbb{R}^{h}$ \\
$\phi_{\mathrm{cost}}$ & $\mathrm{stats}(\cdot) \in \mathbb{R}^{8}$ (cross-carrier summary) & projection to option dim $d_{\mathrm{opt}}=8$ & $\mathbb{R}^{8}$ \\
$g_{\mathrm{dc}}$ & $\mathbf{h}_j$ & 2 hidden layers of width $h$ & $\mathbb{R}$ (scalar logit) \\
$g_{\mathrm{carrier}}$ & $(\mathbf{h}_j,\, \ell^{\mathrm{dc}}_j,\, \mathbf{X}_{\mathrm{cost},j,k}^n)$ with carrier embedding $\parallel$ option projection & 2 hidden layers of width $h$ & $\mathbb{R}$ (scalar logit per $k$) \\
\bottomrule
\end{tabular}
\end{table}

\begin{table}[!ht]
\centering
\caption{Hyperparameter search space and selected configuration for the hierarchical proxy model. Only tuned hyperparameters used in the final model are reported.}
\label{tab:proxy_hyperparameters}
\small
\begin{tabular}{lll}
\toprule
\textbf{Hyperparameter} & \textbf{Search space} & \textbf{Selected best value} \\
\midrule
Learning rate 
& log-uniform $[1.5\times10^{-4},\,6\times10^{-4}]$
& $3.731\times10^{-4}$ \\

Weight decay
& log-uniform $[10^{-6},\,3\times10^{-4}]$
& $1.268\times10^{-4}$ \\

Dropout probability $p$
& uniform $[0.08,\,0.27]$ 
& $0.22$ \\

Cost loss weight
& log-uniform $[0.005,\,0.08]$
& $0.0067$ \\

Constraint loss weight
& $\{0.1,\,0.15,\,0.2,\,0.25,\,0.3\}$
& $0.2$ \\

Carrier loss weight
& log-uniform $[5\times10^{-4},\,10^{-2}]$
& $0.007899$ \\

Carrier embedding dimension
& $\{8,\,16,\,32\}$
& $16$ \\

Option projection dimension
& $\{8,\,16,\,24\}$
& $8$ \\

Number of hidden layers
& $\{2,\,3,\,4\}$
& $2$ \\

Hidden dimension
& $\{128,\,160,\,192,\,224\}$
& $224$ \\

DC embedding dimension
& $\{32,\,64,\,96\}$
& $64$ \\

Batch size
& $\{32,\,64,\,96,\,128\}$
& $96$ \\
\bottomrule
\end{tabular}
\end{table}

\subsection{Decoder Variants for the Inference Strategy Ablation} \label{app:decoder_variants}

The four alternatives to Algorithm~1 compared in Table~5 consume the same logits $(\boldsymbol{\ell}^{\mathrm{dc}}, \boldsymbol{\ell}^{\mathrm{carrier}})$ and return a feasible decision in $\mathcal{Z}^t$; they differ only in the DC-ranking score and eligibility filter, with the per-DC carrier choice $k^*(j) = \arg\max_{k \in \mathcal{K}_j^t} p_{j,k}^{\mathrm{carrier}}$ shared and $K$ denoting the top-$K$ rank cutoff.
\begin{itemize}[noitemsep,topsep=2pt,partopsep=0pt,parsep=0pt]
\item \textbf{Inv-Weighted} (proposed, Algorithm~1): $s_j^t = p_j^{\mathrm{dc}}\, r_j^t$ with $r_j^t = \min\{\mathrm{Inv}_j^t/q^t,\, 1\}$; greedy fill on $\mathcal{J}^{\mathrm{elig}} = \{j : \mathrm{Inv}_j^t > 0,\, \mathcal{K}_j^t \neq \emptyset\}$ in decreasing $s_j^t$.
\item \textbf{Greedy-Prob}: identical to Inv-Weighted but with $s_j^t = p_j^{\mathrm{dc}}$ (the multiplier $r_j^t$ is dropped).
\item \textbf{Top-K-Feasible-DC}: restricts the eligible set to the top-$K$ DCs by $p_j^{\mathrm{dc}}$ that also satisfy $\mathrm{Inv}_j^t \ge q^t$, then greedy-fills on this restricted set; relaxes to $\mathcal{J}^{\mathrm{elig}}$ if no DC qualifies.
\item \textbf{Feasible-DC-First}: first selects the highest-$p_j^{\mathrm{dc}}$ DC satisfying $\mathrm{Inv}_j^t \ge q^t$; falls back to the Inv-Weighted fill if none meets full coverage.
\item \textbf{Top-K-Feasible-Joint}: ranks $(j,k) \in \mathcal{J} \times \mathcal{K}$ pairs by $p_j^{\mathrm{dc}}\, p_{j,k}^{\mathrm{carrier}}$, retains the top $K$, and assigns greedily under $\mathrm{Inv}_j^t > 0$ and $k \in \mathcal{K}_j^t$; multiplying $p_j^{\mathrm{dc}}$ with $p_{j,k}^{\mathrm{carrier}} = p(k \mid j)$ as if independent breaks the hierarchical factorization $f_\theta$ is trained against.
\end{itemize}

\section{Reproducibility} \label{app:repro}
The computational pipeline runs in six stages: preprocessing the public JD.com transaction records; the carrier-service augmentation of Appendix~\ref{app:prep}; training the demand and delivery-time scenario generators; offline C-SAA label generation; proxy training; and online simulation. Fixed random seeds control augmentation, scenario generation, proxy training, and simulation replications; C-SAA scenario streams use order-indexed seeds derived from the global seed. After preprocessing, the study contains $237{,}407$ orders, $1{,}488$ SKUs, $55$ distribution centers, and $17$ carrier services. The proxy is trained on $31{,}907$ peak-hour order-SKU examples with $500$ scenarios per example, and the online test set contains $38{,}416$ peak-hour orders over March 26--31, 2018. Reported runtimes measure policy execution time and exclude scenario generation and post-decision simulator sampling.

\bibliographystyle{plainnat}
\bibliography{refs}

\end{document}